\documentclass[a4paper]{amsart}

\usepackage{amsmath, amssymb, amsfonts, amsthm}
\usepackage{fullpage}
\usepackage{hyperref}
\usepackage[T1]{fontenc}
\usepackage{array}
\usepackage{booktabs}
\usepackage{tabularx}
\usepackage{comment}

\newcommand{\mf}{\mathfrak}
\newcommand{\mc}{\mathcal}
\newcommand{\mb}{\mathbb}
\newcommand{\NN}{\mb N}
\newcommand{\RR}{\mb R}
\newcommand{\ZZ}{\mb Z}
\newcommand{\kk}{\Bbbk}

\newcommand{\beq}{\begin{equation}}
\newcommand{\eeq}{\end{equation}}

\newcommand{\DMO}{\DeclareMathOperator}
\DMO{\ad}{ad}
\DMO{\height}{ht}
\DMO{\weight}{wt}
\DMO{\GK}{GKdim}
\DMO{\gr}{gr}
\DMO{\md}{md}
\DMO{\ann}{Ann}
\DMO{\Aut}{Aut}
\DMO{\hilb}{hilb}
\DMO{\len}{len}
\DMO{\tr}{tr}
\DMO{\spann}{span}
\DMO{\Res}{Res}
\DMO{\Stab}{Stab}

\newcommand{\CC}{\mb C}
\newcommand{\del}{\partial}

\newcommand{\V}{V\!ir}

\newcommand{\hra}{\hookrightarrow}

\theoremstyle{plain}
\numberwithin{equation}{section}
\newtheorem{theorem}[equation]{Theorem}
\newtheorem{proposition}[equation]{Proposition}
\newtheorem{corollary}[equation]{Corollary}
\newtheorem{lemma}[equation]{Lemma}
\newtheorem{sublemma}[equation]{Sublemma}

\theoremstyle{definition}

\newtheorem{remark}[equation]{Remark}
\newtheorem{scholium}[equation]{Scholium}

\title{Ideals in enveloping  algebras of affine Kac-Moody algebras}
\author{Rekha Biswal, Susan J. Sierra}
\address{(Biswal) School of Mathematics,
University of Edinburgh, Edinburgh EH9 3JZ, UK.}
\email{rbiswal@ed.ac.uk}
\address{(Sierra) School of Mathematics,
University of Edinburgh, Edinburgh EH9 3JZ, UK.}
\email{s.sierra@ed.ac.uk}


\begin{document}

\begin{abstract}
Let $L$ be an affine Kac-Moody algebra, with central element $c$, and let $\lambda \in \CC$.
We study two-sided ideals in the central quotient $U_\lambda(L):= U(L)/(c-\lambda)$ of the universal enveloping algebra of $L$, and prove:

\noindent{\bf Theorem 1.}  If $\lambda \neq 0$ then $U_\lambda(L)$ is simple.

\noindent{\bf Theorem 2.} The algebra $U_0(L)$ has  {\em just-infinite growth}, in the sense that any proper quotient has polynomial growth.

%Theorem 2  extends a result of Iyudu and the second named author for the Virasoro algebra.

 As an immediate corollary, we show that the annihilator of any nontrivial  integrable highest weight representation of $L$ is centrally generated, extending a result of Chari for Verma modules.

We also show that universal enveloping algebras of loop algebras and current algebras of finite-dimensional simple Lie algebras have just-infinite growth, and 
prove  similar  results to Theorems 1 and 2 for  quotients of  symmetric algebras of these Lie algebras by Poisson ideals.

\end{abstract}

\maketitle

%\tableofcontents

\section{Introduction}\label{INTRO}

Fix an algebraically closed field $\kk$ of characteristic 0.
Let $L$ be an affine Kac-Moody algebra over $\kk$.  
This paper is concerned with the structure, and particularly the {\em size}, of two-sided ideals of the universal enveloping algebra $U(L)$:  our main theorem shows that such ideals are extremely large, in a sense that we make precise below.  

In the introduction, to simplify discussion, we assume that $L$ is untwisted:  that is, that there is a finite-dimensional simple Lie algebra $\mf g$ so that, as  a vector space, 
\[ L= \mf g[t,t^{-1}]\oplus \kk c \oplus\kk d,\]
where $c$ is central, $d$ is the derivation measuring degree, and $\mf g[t, t^{-1}]$ is the loop algebra of $\mf g$.
(We consider general affine algebras in the body of the paper.)
The derived subalgebra $L'$ of $L$ is 
\[ L'= \mf g[t,t^{-1}]\oplus \kk c \]
and is the unique (up to isomorphism) nontrivial central extension of the loop algebra $\mf g[t, t^{-1}]$.
To emphasise the relationship between $L$ and $\mf g$, we sometimes write $L = \widehat{\mf g}$.

If $\lambda \in \kk$, define $U_\lambda(L) = U(L)/(c-\lambda)$ and $U_\lambda(L') = U(L')/(c-\lambda)$, so $U(\mf g[t,t^{-1}]) \cong U_0(L')$. 
% and $S_\lambda(L) = S(L)/(c-\lambda)$.  
%As $c-\lambda$ is Poisson central,  $S_\lambda(L)$ is a Poisson algebra.
%Likewise define $U_\lambda(L')$ and $S_\lambda(L')$; again, $S_\lambda(L')$ is a Poisson algebra.
%Note that $S_0(L') \cong S(\mf g[t, t^{-1}]^\sigma)$.
We study two-sided ideals in $U(L)$ and in the central quotients $U_\lambda(L)$.   
We will see that they are very big, in a sense we make precise below.
Our main result is:
\begin{theorem}\label{ithm:U}
%For any $\lambda \in \kk$, $U_\lambda(L)$ and $U_\lambda(L')$ have just-infinite growth.  
%That is, let $J$ be a nonzero ideal of $U_\lambda(L)$, and let $J'$ be a nonzero ideal of $U_\lambda(L')$.  
%Then $U_\lambda(L)/J$ and $U_\lambda (L')/J'$ have polynomial growth.
Let $\lambda \in \kk$.
\begin{itemize}
\item[(0)]  The algebras $U_\lambda(L)$ and $U_\lambda(L')$ have just-infinite growth.  
That is, let $J$ be a nonzero ideal of $U_\lambda(L)$, and let $J'$ be a nonzero ideal of $U_\lambda(L')$.  
Then $U_\lambda(L)/J$ and $U_\lambda (L')/J'$ have polynomial growth.
\item[(1)] In fact, if $\lambda \neq 0$ then $U_\lambda(L)$ and $U_\lambda(L')$ are simple rings.
\item[(2)] Any nonzero ideal of $U(L)$ or of $U(L')$ contains a nonzero element of $\kk[c]$; equivalently, the algebras $U(L')\otimes_{\kk[c]}\kk(c)$ and $U(L)\otimes_{\kk[c]}\kk(c)$ are simple.
\end{itemize}
\end{theorem}

Here recall that a $\kk$-algebra $R$ has {\em polynomial growth} if there is a polynomial $p(t) \in \RR[t]$ so that, for every finite-dimensional subspace $V$ of $R$, $\dim V^n \leq p(n)$ for all sufficiently large $n$. 
(For example, enveloping algebras of finite-dimensional Lie algebras have polynomial growth.)
Further, $R$ has {\em exponential growth} if there exists $V \subset R$ so that $\lim_{n\to \infty} (\dim V^n)^{1/n} > 1$ and (strictly) {\em subexponential growth} if $R$ has neither exponential nor polynomial growth.
It is well-known that $U_\lambda(L)$ and $U_\lambda(L')$ have subexponential growth.  
Thus Theorem~\ref{ithm:U}(0) tells us that two-sided ideals in these algebras are extremely large:  large enough to cut these very big algebras down to a reasonable size.

We also prove a similar theorem  for Poisson ideals in the symmetric algebras of $L$ and $L'$. 
Recall that the symmetric algebra $S(\mf k)$ of a Lie algebra $\mf k$ is a Poisson algebra under the Kostant-Kirillov Poisson bracket $\{ - , -\}$ induced by defining $\{ x,y\} = [x,y]$ for $x,y \in \mf k$.
A {\em Poisson ideal} of a Poisson algebra $R$ is an ideal of $R$ which is also a Lie ideal for the Poisson bracket of $R$.
We prove:
\begin{theorem}\label{ithm:S}
Let $\lambda \in \kk$ and  consider the Poisson algebras $S_\lambda(L) = S(L)/(c-\lambda)$ and $S_\lambda(L') = S(L')/(c-\lambda)$.  
\begin{itemize}
\item[(0)]  Let $I$ be a nonzero Poisson ideal of  $S_\lambda(L)$ and let $I'$ be a nonzero Poisson ideal of $S_\lambda(L')$.
Then $S_\lambda(L)/I$ and $S_\lambda(L')/I'$ have polynomial growth.  
\item[(1)] In fact, if $\lambda \neq 0$ then $S_\lambda(L)$ and $S_\lambda(L')$ are {\em Poisson simple} in the sense that they have no nontrivial Poisson ideals.
\item[(2)] Any nonzero Poisson ideal of $S(L)$ or of $S(L')$ contains a nonzero element of $\kk[c]$; equivalently,   $S(L')\otimes_{\kk[c]}\kk(c)$ and $S(L)\otimes_{\kk[c]}\kk(c)$ are Poisson simple.

\end{itemize}
\end{theorem}

As an immediate consequence of  Theorem~\ref{ithm:U} we   compute the annihilators of a large class of representations of $L$, and in particular for {\em all} nontrivial integrable  highest weight representations of $L$: we  show these annilators are all centrally generated.
\begin{theorem}
 \label{ithm:annihilator}
Let  $N$ be a nonzero representation of $L$ so that $(c-\lambda)N=0$ for some $\lambda \neq 0$.  Then $\ann_{U(L)}(N) = (c-\lambda)$.  
In particular, if 
 $M$ is a nontrivial  integrable highest weight represention of $L$ with central character $\lambda$,
then $\ann_{U(L)}(M) = (c-\lambda)$.
\end{theorem}
\begin{proof} Theorem~\ref{ithm:U}(1) shows that $(c-\lambda)$ is a maximal ideal of $U(L)$.  The second statement is an immediate consequence, since nontrivial highest weight representations of $L$ have nontrivial central character.
\end{proof}
Theorem~\ref{ithm:annihilator} extends Vyjayanthi Chari's 
well-known calculation of the annihilators of Verma modules for infinite-dimensional symmetrisable Kac-Moody algebras \cite{Chari}.   
To our knowledge,  no annihilator of a (nontrivial) integrable  highest weight representation of an affine algebra has been known until now.

 Our investigation was partially motivated by work of Natalia Iyudu and the second author \cite{IS}, showing that the analogue of Theorem~\ref{ithm:U}(0) holds for the Virasoro Lie algebra:  central quotients of the enveloping algebra of the Virasoro algebra have just-infinite growth. 
  (See \cite[Theorem~1.2]{IS}.)  
  Affine Kac-Moody algebras (more precisely, their derived subalgebras) and the Virasoro algebra are both central extensions of a graded simple Lie algebra of linear growth, and their representation theory is related through the Sugawara construction. 
   It is thus natural to ask whether results for the Virasoro algebra can be extended to cover affine algebras.  However, affine algebras are much more commutative than the Virasoro algebra, where the centralisers of elements are in general finite-dimensional (in fact two-dimensional); thus, naively, two-sided ideals in their enveloping algebras are smaller.  
It is surprising to find that the two-sided  structure of their enveloping algebras is similarly constrained.  

The phrasing of Theorem~\ref{ithm:U}
 appears  redundant:  of course, since $U_\lambda(L)$ is simple for $\lambda \neq 0$, it has just-infinite growth.  However, it reflects the structure of the paper.  We focus first on proving Theorems~\ref{ithm:U}(0) and \ref{ithm:S}(0).
 Our proof strategy here  is broadly similar to the methods of \cite{IS}, but significantly more delicate because of the commutativity problem:  centralisers in $L$ are large in general and so the adjoint action of $L$ is more difficult to control.
Note that to prove  just-infinite growth for a (left and right) non-noetherian algebra like $U_\lambda(L)$, it is helpful for as many commutators as possible to be nonzero in order to ensure that two-sided ideals are big.

To prove Theorems~\ref{ithm:U}(0) and \ref{ithm:S}(0), we
construct an ordered PBW basis for $U_\lambda(L)$ and then show, through analysing the adjoint action of $L$ on $U_\lambda(L)$, that if $B$ is  a    sufficiently large  element of this basis, then there is an element of $J$ with leading term $B$.  
(See Lemma~\ref{lem:reduction}.)  
This allows us to reduce almost all basis elements, modulo $J$, to smaller elements and thus to bound the growth of $U_\lambda(L)/J$.

We work first with the symmetric algebra of the positive current algebra $t \mf g[t]$ of $\mf g$.
This is the associated graded ring of $U(t \mf g[t])$ under the natural (length) filtration and so has a Poisson algebra structure; further it is finitely graded under the  grading induced by giving elements of $\mf g$ degree 0 and $t$ degree 1.  
We show that if $I$ is a nontrivial Poisson ideal of $S(t\mf g[t])$, then we can reduce almost all monomials $B$ in an ordered PBW basis of $S(t \mf g[t])$ to smaller monomials, modulo $I$.  
This reduction result allows us to prove (Theorem~\ref{thm:current}) that $U(\mf g[t])$ has just-infinite growth.
We then extend our analysis to $U(\mf g[t, t^{-1}])$, to general $U_\lambda(L')$, and to $U_\lambda(L)$.

A corollary of Theorem~\ref{ithm:U}(0) is  the following somewhat counterintuitive result:
\begin{proposition}\label{iprop4}
[Proposition~\ref{prop:Icapcurrent}, Corollary~\ref{cor:cute}]
Let $\mf g$ be a finite-dimensional simple Lie algebra, and let $J$ be a nontrivial ideal of the enveloping algebra $U(\mf g[t,t^{-1}])$ of the loop algebra of $\mf g$.
If $X \subseteq \ZZ$ is any infinite set (for example, $X$ consists of all the primes), and $g \in \mf g \setminus \{0\}$, then $J$ contains a (nonzero) element involving only Lie algebra elements  of the form $g t^x$ for $x \in X$.  

In particular, if $L$ is the (untwisted) affine algebra associated to $\mf g$ and $J$ is a nonzero ideal of $U(L)$, then   $J \cap U(\mf g[t])\neq (0)$.
\end{proposition}

Proposition~\ref{iprop4} allows us to conclude that if $e$ is an element of the nilradical of the positive Borel of $\mf g$, then any nontrivial Poisson ideal of $U_\lambda(L)$ contains a nonzero element $E \in \kk[ \dots, e t^{-1}, e, et, \dots]$.  
By considering the adjoint action of the Cartan subalgebra of $\mf g$ on $E$, we are able to reduce the length of $E$ and thus, by induction prove Theorem~\ref{ithm:U}(1).  Similar techniques give Theorem~\ref{ithm:U}(2) and parts $(1)$, $(2)$ of Theorem~\ref{ithm:S}.

\smallskip

To conclude the introduction, we briefly describe the structure of the paper.   
Section~\ref{BASICS} establishes notation and basic facts about affine Lie algebras and loop algebras, as well as the basics of growth and GK-dimension.
Key   results on Poisson ideals of  symmetric algebras of  positive current algebras are proved in Section~\ref{TWIST}.
In Section~\ref{USEFUL} we prove preparatory results needed to extend our methods from $U(t\mf g[t])$ to $U_\lambda(L)$ before proving Theorems~\ref{ithm:U}(0) and \ref{ithm:S}(0) in Section~\ref{PROOFS}. 
In Section~\ref{SIMPLE} we prove the remaining parts of Theorems~\ref{ithm:U} and \ref{ithm:S}.
Finally, in Section~\ref{APPLICATIONS} we  give several other applications of Theorem~\ref{ithm:U}, including Proposition~\ref{iprop4}.

\smallskip

\noindent{ \bf Acknowledgements: }

The first author is supported, and the second author is partially supported, by the EPSRC grant EP/T018844/1.  
We thank the EPSRC for their support.

We are grateful to 
Lucas Calixto, 
Vyjayanthi Chari,
Travis Scrimshaw, 
and Sankaran Viswanath for helpful discussions and ideas.

\section{Bases and basic facts}\label{BASICS}
If $L$ is an affine Kac-Moody algebra with central element $c$ then there is a finite-dimensional simple Lie algebra $\mf g$ and $\sigma \in \Aut(\mf g)$ so that $L'/(c)$ is isomorphic to the (possibly twisted) loop algebra $ \mf g[t, t^{-1}]^\sigma$.
In this section we establish notation and basic facts about loop algebras, and further establish ordered bases for loop  algebras and for their enveloping/symmetric algebras, which we will use in the next section. 
We also recall some background on growth and Gelfand-Kirillov dimension of algebras and modules.

We adopt the following notation:
%\begin{notation}\label{not:loop}
Let $\mf g$ be a finite dimensional simple Lie algebra.
Let $\sigma \in \Aut(\mf g)$ which (without loss of generality) we assume comes from an automorphism of the Dynkin diagram of $\mf g$.
We also let $\sigma$ denote the diagram automorphism and the associated automorphism of the root system of $\mf g$.
Let $r$ be the order of $\sigma$. 
Let $\eta$ be a primitive $r$'th root of unity. 
We extend the action of $\sigma$ to $\mf g[t, t^{-1}]$ and denote the corresponding twisted current algebra by $\mf g[t]^\sigma$ and the twisted loop algebra by $\mf g[t, t^{-1}]^\sigma$.
%\end{notation}

\subsection{Basic facts about $ \mf g$} \label{ssBASICS}
In this subsection we establish notation and basic facts about $\mf g$.  
 Fix a Cartan subalgebra $\mf h$ of $\mf g$.
If $\alpha \in \mf h^*$, we denote the $\alpha$-eigenspace of $\mf g$ by $\mf g_\alpha$.  

If $s \in \ZZ$, let $\overline s$ denote the congruence class of $s$ in $\ZZ/r\ZZ$. 
The automorphism $\sigma$ induces  a $\ZZ/r\ZZ$-grading $\mf g = \bigoplus_{s=0}^{r-1} \mf g_{\overline s}$ of $\mf g$, where $\mf g_{\overline s} = \{ g \in \mf g : \sigma(g) = \eta^s g \}$.
If $g \in \mf g_{\overline s}$, we write $|g| = \overline s$.
The $\mf{g}_{\overline{s}}$ for $0 \leq s \leq r-1$ are irreducible  $\mf{g}_{\overline{0}}$-modules under the adjoint action.

Let $\mf{h}_{\overline{0}} = \mf g_{\overline 0} \cap \mf h$ be the Cartan subalgebra of $\mf{g}_{\overline{0}}$. 
As in \cite[p. 130]{Kac}, for $\overline s \in \ZZ/r\ZZ$ let  $\Delta_{\overline s}$ be the set of nonzero weights of $\mf h_{\overline 0}$ on $\mf g_{\overline s}$ and define the weight space decomposition
\[ \mf g_{\overline s} = \bigoplus_{\alpha \in \Delta_{\overline s} \cup \{0\}} \mf g_{\overline{s}, \alpha}.\]
We say an element $g \in \mf g$ is {\em weight-homogeneous} if $g$ is in some $\mf g_{\overline s, \alpha}$.
Any weight-homogeneous element is by definition $\sigma$-equivariant.
It is clear that $[\mf{g}_{\overline s, \alpha}, \mf{g}_{\overline p, \beta}]\subseteq \mf{g}_{\overline s+ \overline p, \alpha+\beta}$.

We observe the following basic facts.  As these are standard, most are stated without proof.
\begin{itemize}
\item[(0)]  We have $\mf g[t, t^{-1}]^\sigma = \bigoplus_{s \in \ZZ} \mf g_{\overline s} t^s$.  
    \item[(1)] 
    The automorphism $\sigma$ preserves the set of simple roots of $\mf g$ and so the height of elements of $\mf g$.  
    In particular, the positive Borel subalgebra $\mf b^+$ of $\mf g$ and the nilpotent radical $\mf n^+$ of $\mf b^+$ are $\sigma$-invariant and thus decompose as  sums
    \[ \mf n^+ = \bigoplus_{s = 0}^{r-1} \mf n^+_{\overline s} \]
    and
        \[ \mf b^+ = \bigoplus_{s = 0}^{r-1}  \mf b^+_{\overline s} \]
of $\sigma$-weight spaces.  
 
 \item[(2)] For each $\overline s \in \ZZ/r\ZZ$ and each $\alpha \in \Delta_{\overline s}$ we either have $\mf g_{\overline{s}, \alpha} \subseteq \mf n^+$ 
or $\mf g_{\overline{s}, \alpha} \subseteq \mf n^-$.
 Thus we can define
 \[ \Delta^{+}_{\overline s} = \{ \alpha \in \Delta_{\overline s} : \mf g_{\overline{s}, \alpha} \subseteq \mf n^{+}\}.\]
Let $\Delta^+ = \{ (\overline s, \alpha) : \alpha \in \Delta_{\overline s}^+ \} $.
Similarly, define $\Delta^-_{\overline s}$ and $\Delta^-$. 
Let $\Delta = \Delta^+ \cup \Delta^- = \{(\overline s, \alpha) : \alpha \in \Delta_{\overline s}\}$.

 \item[(3)] Let  $\theta$ be the longest root of $\mf g$. There are $(\overline s, \alpha) \in \Delta^+$ so that $\mf g_\theta = \mf g_{\overline{s}, \alpha} $.

 \item[(4)] For $\alpha \neq 0$ the spaces  $\mf g_{\overline{s}, \alpha}$ are one-dimensional, so we may choose a generator $g_{\overline{s}, \alpha}$.
 Also, if $\gamma $ is a positive roof of $\mf g$, let $e_\gamma, f_\gamma$ be the Chevalley generators, respectively, of $\mf g_\gamma$ and $\mf g_{-\gamma}$.
Let $h_\gamma = [e_\gamma, f_\gamma]$.  

 \item[(5)] 
 If $(\overline p, \beta) \in \Delta^\pm$, then there are $(\overline s_1, \gamma_1), \cdots , (\overline s_k, \gamma_k) \in \Delta^\mp$ so that
 \[[\mf{g}_{\overline s_k, \gamma_k},[\cdots, [ \mf{g}_{\overline s_1, \gamma_1}, \mf{g}_{\pm \theta}]\cdots ]]=\mf{g}_{\overline p,\beta}.\]

\item[(6)] The centralizer of $\mf{h}_{\overline 0}$ in $\mf{g}$ is a Cartan subalgebra of $\mf{g}$ and is thus equal to $\mf h$, and so 
\[\mf h = \mf g_{\overline 0,0} \oplus \dots \oplus \mf g_{\overline{r-1},0},\] 
and $\mf h_{\overline 0} = \mf g_{\overline 0, 0}$.

\item[(7)]  If $(\overline s, \alpha) \in \Delta^+$, then $(\overline{-s}, -\alpha) \in \Delta^-$
 and $[\mf{g}_{\overline s, \alpha},\mf{g}_{\overline{-s}, -\alpha}] \subseteq \mf{h}_{\overline 0}$.
 Let $\kappa$ be the Killing form of $\mf g$, which restricts to a nondegenerate bilinear form on $\mf h_{\overline 0}$, and let $\nu: \mf h_{\overline 0} \to \mf h_{\overline 0}^*$ be the induced isomorphism.
Then $[\mf{g}_{\overline s, \alpha},\mf{g}_{-\overline s,-\alpha}]=\kk \nu^{-1}(\alpha)$.
Further, there is a positive root $\gamma$ of $\mf g$ so that
\[ \mf g_{\overline s, \alpha} \subseteq \sum_{i=0}^{r-1} \mf g_{\sigma^i(\gamma)}.\]

Let $r'$ be the $\sigma$-order of $\gamma$.  Then, 
up to nonzero scalar multiple,
$$(g_{\overline s, \alpha},g_{\overline{- s},-\alpha})=
\begin{cases}
(e_{\gamma}+\eta e_{\sigma(\gamma)}+ \eta^2 e_{\sigma^2(\gamma)},f_{\gamma}+\eta^2 f_{\sigma(\gamma)}+\eta f_{\sigma^2(\gamma)}) \text{ if } r'=3 \text{ and } \overline s= \overline 1, \\
(e_{\gamma}+\eta^2 e_{\sigma(\gamma)}+ \eta e_{\sigma^2(\gamma)},f_{\gamma}+\eta f_{\sigma(\gamma)}+\eta^2 f_{\sigma^2(\gamma)}) \text{ if } r'=3 \text{ and } \overline s= \overline 2, \\
(e_{\gamma}+e_{\sigma(\gamma)}+ e_{\sigma^2(\gamma)},f_{\gamma}+f_{\sigma(\gamma)}+f_{\sigma^2(\gamma)}) \text{ if } r'=3 \text{ and } \overline s=\overline 0, \\
(e_{\gamma}+e_{\sigma(\gamma)},f_{\gamma}+f_{\sigma(\gamma)}) \text{ if } r'=2 \text{ and } \overline s=\overline 0, \\
(e_{\gamma}-e_{\sigma(\gamma)},f_{\gamma}-f_{\sigma(\gamma)})  \text{ if } r' =2 \text{ and } \overline s=\overline 1 ,\\
(e_{\gamma},f_{\gamma}) \text{ if } r'=1 
\end{cases}
$$
Also,  
$$
[g_{\overline s, \alpha},g_{\overline{- s}, -\alpha}]=
\begin{cases}
h_{\gamma}+h_{\sigma(\gamma)}+h_{\sigma^2(\gamma)} \text{ if } r'=3. \\
h_{\gamma}+h_{\sigma(\gamma)} \text{ if } r'=2.\\
h_{\gamma} \text{ if } r'= 1.
\end{cases}
$$
Finally, up to scalars $\{g_{\overline s, \alpha}, g_{\overline{- s}, -\alpha}, [g_{\overline s, \alpha},g_{\overline{- s}, -\alpha}]\}$ forms an $\mf{sl}_2$-triple.

As $\sigma$ preserves the height of elements, each $g_{\overline s, \alpha}$ has a well-defined height in terms of the root system of $\mf g$.
\item [(8)]
When $\sigma \neq \operatorname{id}$ (that is, when $r=2, 3$) we give notation for a $\sigma$-equivariant basis of $\mf h$ in Table~\ref{table1}.
In this case $\mf g$ must be type $A$, $D$, or $E$.
Let the simple roots of $\mf g $ be $\alpha_1, \alpha_2, \dots$ and let 
$h_i=[e_{\alpha_i},f_{\alpha_i}]$ for $1 \leq i \leq \dim \mf{h}$.
We number the simple roots of $\mf g$ as in \cite[Table Fin, p. 53]{Kac}.
\begin{table}[h!]
\begin{center}
\caption{Basis for Cartan subalgebra} \label{table1}
\begin{tabular}{|c|c|c|c|c|}
\hline 
r & $\mf{g}$ & $\mf{h}_{\overline 0}$ & $\mf{h}_{\overline 1}$ & $\mf{h}_{\overline 2}$\\
\hline
2 & $A_{2\ell}, \ell \geq 1$ & $\{ h_i+h_{2\ell-i+1}: 1 \leq i \leq \ell\}$ & $\{ h_i-h_{2\ell-i+1}: 1 \leq i \leq \ell\}$ & 0\\ 
\hline
2 & $A_{2\ell-1}, \ell \geq 1$ & $\{h_i+h_{2\ell-i}: 1 \leq i \leq \ell-1,h_{\ell}\}$ & $\{ h_i-h_{2\ell-i}: 1 \leq i \leq \ell-1\}$ & 0\\
\hline
2 & $D_{\ell+1}, \ell \geq 3$ & $\{h_i: 1 \leq i \leq \ell-1, h_{\ell}+h_{\ell+1}\}$ & $ h_{\ell}-h_{\ell+1}$ & 0\\
\hline
2 & $E_6$ & $h_1+h_5$, $h_2+h_4$, $h_3$, $h_6$ & $h_1-h_5, h_2-h_4$ & 0\\
\hline
3 & $D_4$ & $h_1+h_3+h_4,h_2$ & $ h_1+\eta h_3+\eta^2 h_4$ & $ h_1+\eta^2 h_3 +\eta h_4$\\
\hline
\end{tabular}
\end{center}
\end{table}

\item[(9)]
Let $\mc B_{\Delta} = \{ g_{\overline s, \alpha} : (\overline s, \alpha) \in \Delta\} $ and
 let 
\[ h^i_{\overline s} = \sum_{j=0}^{r'} \eta^{sj} h_{\sigma^j(i)}.\]
(Here $r'$ is the $\sigma$-order of the simple root $\alpha_i$.)
Then
\[ \mc B = \mc B_{\Delta} \cup \{h^i_{\overline s}\}\]
is a $\sigma$-equivariant basis of $\mf g$.
If $0 \leq s \leq r-1$, let $\mc B_{\overline s} = \{ x \in \mc B : |x| =s\}$.
\end{itemize}

\subsection{Two monomial orders on $S(\mf g[t,t^{-1}]^\sigma)$}  \label{ssORDER}

In this section, we define PBW bases of $S(\mf g[t, t^{-1}]^\sigma)$ and $U(\mf g[t, t^{-1}]^\sigma)$.
We will refer to elements of these bases as {\em monomials}, and we will define two different orders on monomials.
The interplay between these monomial orders is key to proving our reduction results.

We first extend the $\sigma$-equivariant basis $\mc B$ of $\mf g$ to a basis of $\mf g[t, t^{-1}]^\sigma$.
Let $\mc B[t, t^{-1}]^\sigma =\{ g t^{r n + |g|} : g \in \mc B, n \in \ZZ\}$, which is a basis of $\mf g[t,t^{-1}]^\sigma$.
Let $\mc B[t]^\sigma = \mc B[t,t^{-1}]^\sigma \cap \mf g[t]^\sigma$.
We will need an order on $\mc B[t, t^{-1}]^\sigma$, which we define as follows.

We first give an order on the roots of $\mf g$.
Let $R$ be the set of roots of $\mathfrak{g}$ and let $R^+, R^-$ be  respectively the sets of positive and negative roots.
Let  $\{\alpha_i: i \in \{1,\cdots,n\}\}$ be the set of simple roots of $\mathfrak{g}$.
If $\alpha=\sum b_i \alpha_i$ is a positive root, then we define the height of $\alpha$ to be $\height \alpha = \sum b_i$.
Recall that we denote the highest root of $\mathfrak{g}$ by $\theta$. 

We will define a total order $<$ on $R$ by:
\begin{itemize}
    \item $\alpha < \beta \text{ if } \height\alpha < \height\beta$.
    \item If $\alpha=\sum b_i \alpha_i$, $\beta=\sum c_i \alpha_i$ and $\height\alpha=\height\beta$, then $\alpha < \beta$ if the tuple $(b_1,\cdots,b_n)$  is less than the tuple $(c_1,\cdots,c_n)$ in lexicographic order.
\end{itemize}
For each $\beta \in R$ 
let $g_\beta = e_\beta$ if $\beta \in R^+$ and $g_\beta = f_{-\beta}$ if $\beta \in R^-$.
Then
\[ \mc C = \{ g_\beta:  \beta \in R\} \cup \{ h_{\alpha_i}: 1 \leq i \leq n\}\]
is a basis for $\mf g$.
Let  $(\beta_1,\cdots, \beta_{\ell})$ be an enumeration of $R^+$ written in increasing order w.r.t. the above total order (so $\beta_\ell = \theta$).
We define the following total order on $\mc C$:
$$g_{-\beta_{\ell}} < \cdots < {g}_{-\beta_{1}} < {h}_{\alpha_1} < \cdots < {h}_{\alpha_n}<
{g}_{\beta_1} < \cdots < {g}_{\beta_{\ell}}.$$

Given $g \in \mf g$, define the {\em $\mc C$-leading term} of $g$, which we denote by $LT_{\mc C}(g)$, to be the largest element of $\mc C$ occurring in $g$ with nonzero coefficient.
(The reason for the subscript $\mc C$ in the notation is to distinguish this from other uses of the terminology ``leading term'' in this paper.)

The ordering defined above on $\mc C$ has the property that:
\beq\label{orderC}
\mbox{ if $x,y,z \in \mc C$ with $[x,y] \neq 0 \neq [x,z]$, then } y < z  \iff  LT_{\mc C}[x,y] < LT_{\mc C}[x, z].
\eeq

We now induce orderings on each  basis $\mc B_{\overline s}$ of $\mf g_{\overline s}$ from our order on $\mc C$.  
Let $0 \leq s \leq r-1$.  
If $x,y \in \mc B_{\overline s}$, define
\[ x<y \quad \iff \quad LT_{\mc C}(x) < LT_{\mc C}(y).\]
(We note here that it  can be seen by inspection that, since every element of $\mc B_{\overline s} $ is associated to a unique $\sigma$-orbit in $\mc C$,  if $LT_{\mc C}(x) = LT_{\mc C}(y)$ then $x=y$.)
For $g \in \mf g_{\overline s}$, we  define $LT_{\mc B}( g )$ to be the largest element of $\mc B$ occurring in $g$ with nonzero coefficient.

It follows from \eqref{orderC} that
\beq\label{orderB}
\begin{split}
\mbox{if $x \in \mc B_{\overline s}$, $y \in \mf g_{\overline s}$ with $LT_{\mc B}(y)< x$, and $z$ is weight-homogeneous with $[z,x] \neq 0 \neq [z,y]$,}\\
\mbox{ then $LT_{\mc B}[z,y]<LT_{\mc B}[z,x]$.}
\end{split}
\eeq

As an example, we give the orderings on the $\mc B_{\overline s}$ in the case $\mf g = A_2$, $|\sigma| = 2$.  
We have
\[ \mc B_{\overline 0}: g_{-\alpha_1} < g_{-\alpha_2} < h_1 +h_2 < g_{\alpha_2} < g_{\alpha_1},\]
\[ \mc B_{\overline 1}:  g_{-\theta} < h_1 -h_2 < g_\theta.\]

 Using the total order $<$ on the $\mc B_{\overline s}$, we define a total order on the basis $\mc B[t, t^{-1}]^\sigma$ of $\mf{g}[t, t^{-1}]^{\sigma}$ as follows:
If $x t^{a}, yt^{b} \in \mc B[t, t^{-1}]^{\sigma}$, then $xt^{a} < yt^{b}$ iff
\begin{itemize}
    \item $a < b$; or
    \item  $a = b$, and $x < y$ in the total order defined on $\mc B_{\overline a}$.
\end{itemize}

We now define a basis of normal words or standard monomials for $S(\mf g[t,t^{-1}]^\sigma)$ and $U(\mf g[t, t^{-1}]^\sigma)$.
We will refer to elements of $\mc B[t, t^{-1}]^\sigma$ as {\em letters}, as ``words'' (or monomials) in these letters will span the algebras of interest.
A {\em standard monomial} or {\em normally ordered monomial} in  the letters $\mc B[t,t^{-1}]^\sigma$ is an expression of the form:
\beq\label{standard} M = g_1 t^{r m_1 + |g_1|} \dots g_k t^{rm_k+ |g_k|} ,\eeq
where the $g_i\in \mc B$, and $g_1 t^{r m_1 + |g_1|} \leq \dots \leq g_k t^{rm_k+ |g_k|}$.
We sometimes write a standard monomial as
\[ M = g_1 t^{n_1} \dots g_k t^{n_k},\]
and when we do so we assume that $\overline{n_a} = |g_a|$ for all $a$: in other words, that each $g_i t^{n_i} \in \mf g[t, t^{-1}]^\sigma$, so $M 
\in  S(\mf g[t, t^{-1}]^\sigma)$.
We say that $k$ is the {\em length} of $M$, which we denote $\len M$, and that the total $t$-power $\sum_a n_a $ is the {\em degree} of $M$, which we denote $\deg(M)$.

We introduce two total orderings on the set of standard monomials. 
For two standard monomials $M_1$ and $M_2$, we write $M_1 < M_2$ if 
\begin{itemize}
    \item $\len M_1 < \len M_2$ or 
    \item $\len M_1= \len M_2$ and $\deg M_1 < \deg M_2$ or
    \item $\len M_1 = \len M_2$, $\deg M_1 = \deg M_2$, and $M_1$ is less than $M_2$ with respect to  the {\em left to right} lexicographic order when both $M_1$ and $M_2$ are written in increasing (that is, normal) order.
\end{itemize}
Similarly, we write $M_1 \prec M_2$ if
\begin{itemize}
    \item $\len M_1 < \len M_2$ or 
    \item $\len M_1 =\len M_2$ and $\deg M_1 < \deg M_2$ or
    \item $\len M_1=\len M_2$, $\deg M_1 = \deg M_2$, and $M_1$ is less than $M_2$ with respect to the  {\em  right to left} lexicographic order when both $M_1$ and $M_2$ are written in increasing order.
\end{itemize}

By the PBW theorem, both $U(\mf g[t, t^{-1}]^\sigma)$ and $S(\mf g[t, t^{-1}]^\sigma)$ have a basis of standard monomials.
Given nonzero $F\in U(\mf g[t, t^{-1}]^\sigma)$ or $F \in S(\mf g[t, t^{-1}]^\sigma)$, we define $LT_< F $, respectively $LT_\prec F$, to be the $<$-largest, respectively $\prec$-largest, standard monomial occurring in $F$ with nonzero coefficient when $F$ is written as a linear combination of standard monomials.

\subsection{Growth} \label{ssGROWTH}
In this short subsection we recall some definitions and basic facts about the growth of algebras and modules.  
Let $R$ be a finitely generated associative $\kk$-algebra and let $M$ be a finitely generated representation of $R$.
Then  $R$ has {\em polynomial growth} if there is a polynomial $p(t) \in \RR[t]$ so that, for every finite-dimensional subspace $V$ of $R$, $\dim V^n \leq p(n)$ for all sufficiently large $n$. 
Further, $R$ has {\em exponential growth} if there exists $V \subset R$ so that $\lim_{n\to \infty} (\dim V^n)^{1/n} > 1$ (note that this limit always exists) and (strictly) {\em subexponential growth} if $R$ has neither exponential or polynomial growth.
The growth of $M$ is defined similarly; here we let $V \subset R$ and $ W \subset M$ be finite-dimensional vector spaces and consider $\dim V^n W$ as $n \to \infty$.
If $\mf k$ is an infinite-dimensional Lie algebra of polynomial growth (for example, an affine or loop algebra) then  $U(\mf k)$ has subexponential growth by  \cite{Smith}.

The idea of growth may be refined through considering the {\em Gelfand-Kirillov dimension} or $\GK$ of $R$ and of $M$.
Here we define
\[ \GK R = \sup_V \varlimsup \log_n \dim V^n,\]
where the supremum is taken over all finite-dimensional subspaces $V$ of $R$.
Likewise,
\[ \GK M = \sup_{V, W} \varlimsup \log_n \dim V^n W,\]
where the supremum is taken over all finite-dimensional subspaces $V$ of $R$ and $W$ of $M$.
If $R$ is a finitely generated $\kk$-algebra and $M$ is a finitely generated $R$-module, let $V$ be any finite-dimensional generating subspace of $R$ that contains 1, and let $W$ be any finite-dimensional generating subspace of $M$.
Then
\[ \GK R =  \varlimsup \log_n \dim V^n, \quad \GK M =  \varlimsup \log_n \dim V^n W,\]
and does not depend on the choice of $V$ or $W$.

Note that $R$ (respectively, $M$) has polynomial growth if and only if $\GK R < \infty$ (respectively, $\GK M < \infty$).
We say that $R$ has {\em just-infinite growth} (equivalently, {\em just-infinite GK-dimension}) if $\GK R = \infty$ but for any nontrivial ideal $J$ of $R$, then $\GK R/J < \infty$.
For more background about growth and GK-dimension, we refer the reader to \cite{KL}.

If $\mf k$ is a Lie algebra, then $\GK U(\mf k) = \dim \mf k$.  
Further, by \cite{Smith}, if $L$ is an affine Kac-Moody algebra, then the central  quotients $U_\lambda(L)$ and $U_\lambda(L')$ have intermediate growth.

We will repeatedly use the following basic fact about growth of algebras:
.
\begin{scholium}\label{scholium}
Let $K$ be a field and let $A \subseteq B$ be $K$-algebras.
Suppose that $\GK A = \infty$ and that $J $ is an ideal of $B$ such that $B/J$ has polynomial growth.
Then $J\cap A \neq (0)$.
\end{scholium}
\begin{proof}
This follows  directly from the fact that GK-dimension does not increase on subalgebras \cite[Lemma~3.1]{KL}, so we cannot have $A \hra B/J$.
\end{proof}

\section{A reduction result for symmetric algebras of current algebras}\label{TWIST}

Let $\mf g$ be a finite-dimensional simple Lie algebra and let $\sigma \in\Aut(\mf g)$.
We adopt the notation of Section~\ref{BASICS}; in particular, let $r$ be the order of $\sigma$.
The {\em positive twisted current algebra} of $\mf g$ is the subalgebra $(t g[t])^\sigma$ of the (twisted) loop algebra $\mf g[t, t^{-1}]^\sigma$.  
In this section, we consider the symmetric algebra of the positive twisted current algebra of $\mf g$ and show that it has just-infinite growth as  a Poisson algebra:  any factor by a proper Poisson ideal has polynomial growth.
We will generalise this result in later sections to prove Theorems~\ref{ithm:U} and \ref{ithm:S}.

The symmetric algebra $S(\mf g[t, t^{-1}]^\sigma)$ is a Lie algebra under the Poisson bracket $\{ -, -\}$ induced from the Lie bracket on $\mf g[t, t^{-1}]^\sigma$.
That is, for $x,y \in \mf g[t, t^{-1}]^\sigma$ we have $\{x,y\} = [x,y]$, and $\{-,-\}$ is anticommutative and a derivation in each input.

A {\em Poisson ideal} of $S(\mf g[t, t^{-1}]^\sigma)$ is an ideal of the underlying commutative algebra which is also a Lie ideal for the Poisson bracket.
Note that a Lie ideal of $S(\mf g[t, t^{-1}]^\sigma)$ is also  a $\mf g[t, t^{-1}]^\sigma$-subrepresentation, where $\mf g[t, t^{-1}]^\sigma$ acts on $S(\mf g[t, t^{-1}]^\sigma)$ by the adjoint action $\ad gt^a = \{ gt^a, -\}$.
Each $S^m(\mf g[t, t^{-1}]^\sigma)$ is a $\mf g[t, t^{-1}]^\sigma$-subrepresentation of $S(\mf g[t, t^{-1}]^\sigma)$, although of course not a Poisson ideal.

Thus if $I$ is a Poisson ideal of $S(\mf g[t, t^{-1}]^\sigma)$, each $I \cap S^m(\mf g[t, t^{-1}]^\sigma)$ is a $\mf g[t, t^{-1}]^\sigma$-subrepresentation of $S(\mf g[t, t^{-1}]^\sigma)$.  
We will spend the bulk of this section analyzing the structure of such subrepresentations.
In fact, because we want our results to apply to Poisson ideals of current algebras, we will consider the structure of $( t \mf g[t])^\sigma$-subrepresentations of the $m$'th symmetric power $S^m(\mf g[t, t^{-1}]^\sigma)$.

Our key technical result is the following:
\begin{proposition}\label{newcor6}
Let $I$ be any nonzero $(t \mf g[t])^\sigma$ sub-representation of $S^m(\mf g[t, t^{-1}]^\sigma)$. 
There   are $n, \ell \in \ZZ$ 
such that for any  standard monomial $M= g_1 t^{i_1} \dots g_m t^{i_m}$  with $i_1 \geq n$, there is $H_M\in I$ such that $LT_{<}H_M =M$ and  all $t$-powers in $H_M$ are $\geq \ell$.
\end{proposition} 
This result allows us to reduce ``big'' monomials modulo $I$ to a linear combination of smaller monomials.  
The proof is based on the interplay between the $<$ and $\prec$-orders on standard monomials, generalising the key arguments of \cite{IS} to work in our substantially more commutative context.

\subsection{Preparatory results} \label{ssPREP}
Before proving the main reduction result Proposition~\ref{newcor6}, we will need several preparatory lemmata which will help us control the $\prec$-leading terms of elements of $I$.
If $M = g_1 t^{i_1} \dots g_m t^{i_m}$ is a standard monomial, we say that $(g_1, \dots, g_m)$ is the {\em congruence class} of $M$ and write $M \equiv (g_1, \dots, g_m)$.
We will construct, for any congruence class $\underline g \in (\mc B_{\Delta})^m$, an element $H_{\underline g} \in I$ with $LT_{\prec} H_{\underline g} \equiv \underline g$.
We will see that the existence of these elements is sufficient to prove Proposition~\ref{newcor6}.

Constructing the $H_{\underline g}$ will take several steps.
We first show that $H_{\underline g}$ exists for $\underline g = (g_\theta, \dots, g_\theta)$ or $(g_{-\theta}, \dots, g_{-\theta})$; and further, we can assume, in this case, that  all monomials of $H_{\underline g}$ have the same congruence class.

\begin{lemma}\label{lem:sl2theta-twist}
Let $ 0 \neq F \in S^m(\mf g[t,t^{-1}]^\sigma)$, and let
$I$ be the $(t\mf g[t])^\sigma$-subrepresentation of $S^m(\mathfrak{g}[t, t^{-1}]^{\sigma})$ generated by  $F$.
\begin{enumerate}
\item[$(1)$] If $\{g_{\overline{s}, \alpha} t^{rp+s}, F\}=0$ 
for all  $p \in \ZZ_{\geq 1}$ and all $(\overline s, \alpha) \in \Delta^+$
then 
$F \in S^m(g_\theta[t, t^{-1}])$: that is, 
$F$ is a linear combination of monomials of the form ${g}_{\theta}t^{k_1}\cdots {g}_{\theta}t^{k_m}$ for some $k_1 \leq \cdots \leq k_m$.
\item[$(2)$] If $\{g_{\overline{s}, \alpha} t^{rp+s}, F\}=0$ 
for all  $p \in \ZZ_{\geq 1}$ and all $(\overline s, \alpha) \in \Delta^-$
then 
$F \in S^m(g_{-\theta}[t, t^{-1}])$: that is, 
$F$ is a linear combination of monomials of the form ${g}_{-\theta}t^{k_1}\cdots {g}_{-\theta}t^{k_m}$ for some $k_1 \leq \cdots \leq k_m$.
\item[$(3)$] 
There are nonzero $G \in I \cap S^m(g_\theta[t,t^{-1}])$ and $G' \in I \cap S^m(g_{-\theta}[t,t^{-1}]) $.
\end{enumerate}
\end{lemma}
\begin{proof}
   $(1)$  Let $(\overline s, \alpha) \in \Delta^+$.
   Let us take $q \in \ZZ_{\geq 1}$ such that $\overline q =\overline s$ and so that
   if $g_1 t^{j_1} \cdots g_m t^{j_m}$ is a monomial in $F$, then $q > |j_1| + \cdots + |j_m|$.
    Thus, if $M = g_1 t^{j_1}\cdots g_m t^{j_m}$ is a monomial appearing in $F$, then 
    \begin{multline*}\{{g}_{\overline s, \alpha} t^{rq+s}, M\}=g_2 t^{j_2}\cdots g_m t^{j_m}[{g}_{\overline s, \alpha}, g_1]t^{j_1+q}+g_1 t^{j_1}g_3 t^{j_3}\cdots g_m t^{j_m}[{g}_{\overline s, \alpha}, g_2]t^{j_2+q}+\cdots\\
   + g_1 t^{j_1}\cdots g_{v-1}t^{j_{v-1}}[{g}_{\overline s, \alpha}, g_m]t^{j_m+q}.
    \end{multline*}
    Note that  each of the monomials above are written in normal order. 
    Since we have taken $q$ to be large enough, none of those monomials will appear in expressions coming from the action of ${g}_{\overline s, \alpha}t^{q}$ on any  monomial $M' \neq M$ occurring in $F$. 
    Hence by hypothesis the Lie brackets  $[{g}_{\overline s, \alpha}, g_i]$ are all zero.
    This is true for all  $(\overline s, \alpha) \in \Delta^+$ which implies that $[\mf{n}^{+},g_i]=0$ for $1 \leq i \leq m$.
    Hence $g_1=\cdots=g_m=g_\theta$, which spans the centre of $\mf n^+$.
    The proof of $(2)$ is similar.
    
    To prove $(3)$, we first construct $0 \neq G \in I$ so that
    \beq\label{eq:G}
    \{ g_{\overline s, \alpha} t^{rp+s}, G \} = 0 \mbox{ for all $p \in \ZZ_{\geq 1}$ and $(\overline s, \alpha) \in \Delta^+$.}
    \eeq
If $\eqref{eq:G}$ holds for $G=F$, we are done.
Otherwise  there exist $a \in \ZZ_{\geq 1}$ and $(\overline u, \beta) \in \Delta^+$ such that $F_1=\{g_{\overline u,\beta}t^{ra+u},F\}\neq 0$.
   If \eqref{eq:G} holds for $G=F_1$, we are done; else
  there exist $b \in \ZZ_{\geq 1}$ and $(\overline v, \gamma) \in \Delta^+$ such that $F_2=\{g_{\overline v,\gamma}t^{rb+v},F_1\}\neq 0$.
    As $\{g_{\overline s, \alpha}, - \}$ increases the height of monomials and the height of a monomial in $S^m(\mf g[t, t^{-1}]^{\sigma})$ is at most $m \height \theta$, this process must terminate.
    Thus, continuing this procedure repeatedly, we will reach  a nonzero element $G$ of $I$ such that 
    \eqref{eq:G} holds. 
    
    By symmetry there is $0 \neq G' \in I$ such that 
    $\{g_{\overline{s}, \alpha} t^{rp+s}, G'\} = 0$ for all 
    $p \in \ZZ_{\geq 1}$ and $(\overline s, \alpha) \in \Delta^-$.
  The two parts of  $(3)$ then follow by applying $(1)$ and $(2)$.
\end{proof}

Enumerate  $\Delta^+ = \{(\overline \xi_1,\beta_1), \cdots, (\overline \xi_{\ell},\beta_\ell)\}$ where $0 \leq \xi_i \leq r-1$. 
Then the set $\{{g}_{\overline \xi_1, \beta_1}, \dots, {g}_{\overline \xi_{\ell},\beta_\ell} \}$ forms a basis of $\mf{n}^{+}$ and the set $\{{g}_{\overline{- \xi_1},-\beta_1}, \dots, {g}_{\overline{- \xi_{\ell}},-\beta_\ell}\}$ forms a basis of $\mf{n}^{-}$. 
Recall that if $g \in \mf g_{\overline s}$, then we write $\overline s = | g |$.

We now prove the existence of  $H_{\underline g}$  with $LT_\prec H_{\underline g} \equiv \underline g$ for any $\underline g \in (\mc B_{\Delta})^m$.

\begin{lemma}\label{lem:arbitrary congruence classes-twist}
Let $I$ be the $(t \mf g[t])^\sigma$-subrepresentation of $S^m(\mathfrak{g}[t, t^{-1}]^{\sigma})$ generated by an element $F \neq 0$.
Let  $(g_1,\cdots, g_m) \in (\mc B_{\Delta})^m$.
Then there exists $G$ in $I$ such that  $LT_{\prec}(G) \equiv (g_1,\cdots,g_m)$.
\end{lemma}

\begin{proof}
Let $d = \height(\theta)>0$.  Write $\mf{g}_{\theta}=\mf{g}_{\overline s, \alpha}$, which implies that $\mf{g}_{-\theta}=\mf{g}_{\overline{- s}, -\alpha}$ i.e. $|\mf{g}_{\theta}|=-|\mf{g}_{-\theta}|$ (mod $r$). 
We prove by induction that for all $n \in \{0, \dots, m\}$ there is $G(n) \in I$ with the following properties:
\begin{enumerate}
    \item[(1)] $LT_\prec(G(n)) = g_1 t^{k_1} \dots g_n t^{k_n} g_\theta t^{k_{n+1}} \dots g_\theta t^{k_m}$ for some $k_1 \leq \cdots \leq k_m \in \ZZ$;
    \item[(2)] $k_i + 2rd \leq k_{i+1}$ for all $i>n$;  
    \item[(3)] $k_i < k_{i+1}$ for all $i \leq n$.  
\end{enumerate}
Then $G=G(m)$ is the element we seek.

To construct $G(0)$ we use Lemma~\ref{lem:sl2theta-twist}.  
By that result, there is a nonzero $H \in I \cap S^m(\mf g_{-\theta}[t,t^{-1}])$.
Write $LT_\prec H = g_{-\theta} t^{j_1} \dots g_{-\theta} t^{j_m}$ where $\overline j_1 = \cdots = \overline j_m = |\mf{g}_{-\theta}|$.\\
Let $G(0) = (\ad g_\theta t^{rd+|\mf{g}_{\theta}|})^2(\ad g_\theta t^{{2rd+|\mf{g}_{\theta}|}})^2\dots (\ad g_\theta t^{mrd+|\mf{g}_{\theta}|})^2(H)$.
The monomials in $G(0)$ are obtained from monomials in $H$ by applying exactly two $g_\theta t^i$ to each $g_{-\theta} t^{j_k}$. 
(This is because $(\ad g_\theta)^3 = 0$.)
The $\prec$-leading term of $G(0)$ is obtained from $LT_\prec(H)$ by applying the highest powers of $t$ to the rightmost letters.
Thus 
$ LT_\prec(G(0)) = g_\theta t^{k_1} \dots g_\theta t^{k_m}$
where $k_i = j_i + 2(ird+|\mf{g}_{\theta}|)$.  
It is immediate that this has the claimed properties.

Now assume we have constructed $G(n)$ as claimed for some $1\leq n \leq m-1$; we construct $G(n+1)$.  
There are two cases, depending on whether $g_{n+1}$ lies in  a positive or negative root space.

\smallskip
\noindent{\bf Negative case:}  $g_{n+1} \in \mf n^-$.

{\bf Step $1^-$: } let $H_1^- = (\ad g_{-\theta} t^{r-|\mf{g}_{\theta}|})^{2(m-n)}(G(n))$ and let 
\[N_1^- = g_1 t^{k_1} \dots g_n t^{k_n} g_{-\theta} t^{k_{n+1}+2(r-|\mf{g}_{\theta}|)}\dots g_{-\theta} t^{k_m+2(r-|\mf{g}_{\theta}|)}.\]
We claim that $N_1^- = LT_\prec(H_1^-)$.
This is because the $\prec$-leading term of $H_1^-$ is obtained from $LT_\prec(G(n))$ by increasing the right-most powers of $t$ as much as possible, and this clearly gives  $N_1^-$.

{\bf Step $2^-$:} Let $H_2^- = (\ad g_\theta t^{r+|\mf{g}_{\theta}|})^{2(m-n-1)}(H_1^-)$ and  let 
\[N_2^- = g_1 t^{k_1} \dots g_n t^{k_n} g_{-\theta} t^{k_{n+1}+2(r-|\mf{g}_{\theta}|)} g_\theta t^{k_{n+2}+4r} \dots g_{\theta} t^{k_m+4r}.\]
We claim that $N_2^- = LT_\prec(H_2^-)$.

First, we show that $N_2^- = LT_\prec(\ad g_{-\theta} t^{r - |g_\theta|})^{2(m-n})(N_1^-)$.
This is similar to the previous paragraph:  again, the $\prec$-leading term is obtained by increasing the right-most powers of $t$ as much as possible, giving $N_2^-$.

We now show that $N_2^- = LT_\prec(H_2^-)$.
Let $N_2' \neq N_2^-$ be any other monomial of $H_2^-$.
Suppose that $N_2' \in \mf g  t^{k_1} \dots \mf g t^{k_n} \mf g t^{k_{n+1}+2(r-|\mf{g}_{\theta}|)}\mf g t^{k_{n+2}+4r} \dots \mf g  t^{k_m+4r}$ --- that is, $N_2'$ involves the same $t$-powers as $N_2^-$.  
(If this is not the case, it is clear that $N_2' \prec N_2^-$.)
By construction, there must be a monomial   $M = h_1 t^{k_1} \dots h_m t^{k_m}$ of $G(n)$ so that $N_2'$ is a monomial in 
$(\ad g_\theta t^{r-|\mf{g}_{-\theta}|})^{2(m-n-1)}( \ad g_{-\theta} t^{r-|\mf{g}_{\theta}|})^{2(m-n)}(M)$; that is, $N_2'$ is obtained from a monomial in $G(n)$ with the same $t$-powers as $LT_\prec G(n)$.
Further, $N_2'$ must have been obtained from $M$ by applying $g_{-\theta} t^{r-|\mf{g}_{\theta}|}$ twice to  each letter of $M$ in the $n+1$'st place or further right, and $g_{\theta } t^{r+|\mf{g}_{\theta}|}$ twice to each of the letters in the $n+2$'nd place and further right; any other combination of applying $g_{-\theta} t^{r-|\mf{g}_{\theta}|} $ and $g_\theta t^{r+|\mf{g}_{\theta}|}$ would give either the wrong $t$-powers or a zero result.
The only possibility to obtain a nonzero result is if $M = h_1 t^{k_1} \dots h_n t^{k_n} g_\theta t^{k_{n+1} } \dots g_\theta t^{k_m}$, 
as we needed $(\ad g_{-\theta}t^{r-|\mf{g}_{\theta}|})^2$ to be nonzero on the $n+1$'st and greater letters of $M$.
Thus we have $N_2' = h_1  t^{k_1} \dots h_n t^{k_n} g_{-\theta} t^{k_{n+1}+2(r-|\mf{g}_{\theta}|)} g_\theta t^{k_{n+2}+4r} \dots g_\theta  t^{k_m+4r}$, where necessarily 
$h_1  t^{k_1} \dots h_n t^{k_n}\prec g_1  t^{k_1} \dots g_n t^{k_n}$.
So $N_2' \prec N_2^-$.

{\bf Step $3^-$:}  As $\theta$ is the longest root of $\mf g$, by Fact (5) there are 
$p \in \ZZ_{\geq 0}$ and $(\overline{s_1},\eta_1), \dots, (\overline{s_p}, \eta_p) \in \Delta^+$
so that $g_{n+1}$ is a nonzero multiple of 
\[ [g_{\overline s_1,\eta_1}, [\dots, [ g_{\overline s_p, \eta_p}, g_{-\theta}]\dots]].\]
Let us denote $H_3^- = \{g_{\overline s_1,\eta_1}t^{s_1}, \{\dots, \{ g_{\overline s_p,\eta_p}t^{s_p}, H_2^-\}\dots\}\}$ and $S=\sum_{i=1}^p s_i$ and let
\[N_3^- = g_1 t^{k_1} \dots g_n t^{k_n} g_{n+1} t^{k_{n+1}+2(r-|\mf{g}_{\theta}|)+S} g_\theta t^{k_{n+2}+4r} \dots g_\theta t^{k_m+4r}.\]
We claim that $N_3^- = LT_\prec(H_3^-)$.
First, $LT_\prec(\{g_{\overline s_1,\eta_1}t^{s_1}, \{\dots, \{ g_{\overline s_p,\eta_p}t^{s_p}, N_2^-\}\dots\}\})$ is obtained from $N_2^-$ by applying the elements $g_{\overline s_i, \eta_i}t^{s_i}$ as far to the right as possible.  
As $[g_{\overline s_i, \eta_i}, g_\theta] = 0$ for all $i$ (this is because $\theta$ is the highest root of $\mf g$), applying the $g_{\overline s_i,\eta_i}t^{s_i}$ as far to the right as possible means applying them to the $(n+1)$'st letter $g_{-\theta} t^{k_{n+1}+2(r-|\mf{g}_{\theta}|)}$ to obtain $N_3^-$. 
In other words, 
\[N_3^- = LT_\prec(\{g_{\overline s_1,\eta_1}t^{s_1}, \{\dots, \{ g_{\overline s_p,\eta_p}t^{s_p}, N_2^-\}\dots\}\}).\]

Now let $N'_2 \neq N_2^-$ be any other monomial of $H_2^-$ and let $H_3' = \{g_{\overline s_1,\eta_1}t^{s_1}, \{\dots, \{ g_{\overline s_p,\eta_p}t^{s_p}, N'_2\}\dots\}\}$. 
If $N'_2 \not \in \mf g t^{k_1} \dots \mf g t^{k_n} \mf g t^{k_{n+1}+2(r-|\mf{g}_{\theta}|)} \mf g t^{k_{n+2}+4r} \dots \mf g t^{k_m+4r}$, then  we must have $LT_\prec H'_3 \prec N_3^-$ by comparing $t$-powers.  

Suppose now that $N_2'$ involves the same $t$-powers as $N_2^-$: that is, \[N'_2 \in \mf g t^{k_1} \dots \mf g t^{k_n} \mf g t^{k_{n+1}+2(r-|\mf{g}_{\theta}|)} \mf g t^{k_{n+2}+4r} \dots \mf g t^{k_m+4r}.\]
We have seen in the proof of Step $2^-$ that we must have \[N_2' = h_1  t^{k_1} \dots h_n t^{k_n} g_{-\theta} t^{k_{n+1}+2(r-|\mf{g}_{\theta}|)} g_\theta t^{k_{n+2}+4r} \dots g_\theta  t^{k_m+4r}\] for some $h_1,\dots, h_n \in \mc B$.   
Then $LT_\prec(H_3') = h_1  t^{k_1} \dots h_n t^{k_n} g_{n+1} t^{k_{n+1}+2(r-|\mf{g}_{\theta}|)+S} g_\theta t^{k_{n+2}+4r} \dots g_\theta  t^{k_m+4r}$.   
But as $N_2' \prec N_2^-$, thus $h_1  t^{k_1} \dots h_n t^{k_n}\prec g_1  t^{k_1} \dots g_n t^{k_n}$:  that is, $LT_\prec(H_3') \prec N_3^-$.  
Thus $N_3^- = LT_\prec H_3^-$, as claimed.

Let $G(n+1) = H_3^-$.  We have shown that $(1)$ is satisfied; but the $t$-powers in $N_3^-$ also  satisfy $(2)$ and $(3)$.

\smallskip

\noindent{\bf Positive case: } $g_{n+1} \in \mf n^+$.

{\bf Step $1^+$:} Let $H_1^+ = (\ad g_{-\theta} t^{r-|\mf{g}_{\theta}|})^{2(m-n-1)}(G(n))$ and let 
\[ N_1^+ = g_1 t^{k_1} \dots g_n t^{k_n} g_\theta t^{k_{n+1}} g_{-\theta} t^{k_{n+2}+2(r-|\mf{g}_{\theta}|)} \dots g_{-\theta} t^{k_m+2(r-|\mf{g}_{\theta}|)}.\]
We claim that $N_1^+ = LT_\prec H_1^+$.
This is because the $\prec$-leading term of $H_1^+$ must come from the $\prec$-leading term of $G(n)$ by increasing the rightmost $t$-powers as much as possible, which gives $N_1^+$.

{\bf Step $2^+$:} As $\theta$ is the longest root of $\mf g$, by Fact (5) there are 
$q \in \ZZ_{\geq 0}$ and $(\overline{s_1}, \gamma_1), \dots, (\overline{s_q}, \gamma_q) \in \Delta^-$
so that $g_{n+1}$ is a nonzero multiple of 
\[ [g_{\overline s_1,\gamma_1}, [\dots, [ g_{\overline s_q,\gamma_q}, g_{\theta}]\dots]].\]
Let 
$H_2^+ = \{g_{\overline s_1,\gamma_1}t^{r+s_1}, \{\dots, \{ g_{\overline s_q, \gamma_q}t^{r+s_
q}, H_1^+\}\dots\}\}$.
Let $S = qr+\sum_{i=1}^q s_i$ 
and let
\[N_2^+ = g_1 t^{k_1} \dots g_n t^{k_n} g_{n+1} t^{k_{n+1}+S} g_{-\theta} t^{k_{n+2}+2(r-|\mf{g}_{\theta}|)} \dots g_{-\theta} t^{k_m+2(r-|\mf{g}_{\theta}|)}.\]
We claim that $N_2^+ = LT_\prec(H_2^+)$.
This is because, again, the $\prec$-leading term of $H_2^+$ comes by increasing the rightmost $t$-powers as much as possible.
As $[g_{\overline s_i, \gamma_i}, g_{-\theta}] = 0$ for all $i$, this is done by acting on $g_{\theta} t^{k_{n+1}}$ with all of the $g_{\overline s_i,\gamma_i} t^{r+s_i}$ to obtain $N_2^+$. 

Note that $q\leq d$ and each $s_i \leq r$.
Thus $k_{n+1}+S< k_{n+1}+2rd \leq k_{n+2} < k_{n+2} + 2(r-|\mf{g}_{\theta}|)$, using the induction hypothesis.
Therefore $N_2^+$ is a standard monomial.

{\bf Step $3^+$:}
Let $H_3^+ = (\ad g_{\theta}t^{r+|\mf{g}_{\theta}|})^{2(m-n-1)}(H_2^+)$ and let 
\[ N_3^+ = g_1 t^{k_1} \dots g_n t^{k_n} g_{n+1} t^{k_{n+1}+S} g_\theta t^{k_{n+2}+4r} \dots g_\theta t^{k_m+4r}.\]
We claim that $N_3^+ = LT_\prec(H_3^+)$.
This is because the $\prec$-leading term of $H_3^+$ comes from applying $\ad g_\theta t^{r+|g_\theta|}$ as far to the right as possible:
that is, applying $g_\theta t^{r+|\mf{g}_{\theta}|}$ twice to each $g_{-\theta} t^{k_i+2(r-|\mf{g}_{\theta}|)}$, where $i \in \{n+2, \dots, m\}$.

Let $G(n+1) = H_3^+$.
We have verified that $LT_\prec G(n+1)$ has the claimed congruence class, so $(1)$ is satisfied.  
The observation at the end of Step $2^+$ ensures that $(3)$ holds; and $(2)$ follows as $k_{i+1}+4r - (k_i +4r) = k_{i+1}-k_i$.
\end{proof}

Let $I$ be a nontrivial $\mf g[t]^\sigma$-subrepresentation of $S^m(\mf g[t, t^{-1}]^\sigma)$.
We have seen that for any $\underline g \in (\mc B_{\Delta})^m$ there is $H_{\underline g} \in I$ with $LT_\prec H_{\underline g} \equiv \underline g$.
We will use this to construct elements of $I$ with arbitrary $<$-leading term, as long as all $t$-powers involved are sufficiently large.

The next result considers which $<$-leading terms may be obtained from a particular $H_{\underline g}$.  
It is modelled on the methods of \cite{IS}.

\begin{lemma}\label{lem:W-technique-twist}
Let $(g_1, \dots, g_m) \in (\mc B_{\Delta})^m$.  
Let $G \in S^m(\mathfrak{g}[t,t^{-1}]^{\sigma})$  with $LT_{\prec}G=g_1 t^{a_1}\cdots g_mt^{a_m}$. 
Let $I$ be the $(t \mathfrak{g}[t])^\sigma$-subrepresentation of $S^m(\mathfrak{g}[t,t^{-1}]^\sigma)$  generated by $G$. 
Let $s$ be the smallest power of $t$ occuring in $G$.

Suppose that $(g_1', \dots, g_m') \in \mc B^m$ so that
\beq\label{magic}
\mbox{for all $j$ there is weight-homogeneous $ g_j'' \in \mf g$ so that $[g_j'',g_j]$ is a nonzero scalar multiple of $g'_{m+1-j}$.}
\eeq
Then for all standard monomials
 $M=g'_1 t^{i_1}\cdots g'_m t^{i_m}$ 
with $i_1 > \max(0, 2a_m-s)$,
there is $H_M \in I$ with $LT_{<}H_M= M$.
\end{lemma}
\begin{proof}

Let $$H_M=\{g_1''t^{i_m-{a_1}},\{g_2'' t^{i_{m-1}-a_2}\cdots, \{g_m''t^{i_1-a_m}, G\}\cdots \}\}.$$
Then $H_M \in I$ because 
each $g_j'' t^{i_{m+1-j}-a_j} \in \mf g[t]^\sigma$: we have 
$\overline a_j =|g_j|$ and $\overline i_j = |g'_j|$ implies that $\overline {i_{m+1-j}-a_j}=|g_j''|$.
We claim that $LT_{<}H_M=M$.

To prove our claim, let us consider the process by which monomials of $H_M$ are produced from monomials of $G$.
Let 
\[A= h_1 t^{b_1} \cdots h_m t^{b_m}\]
be a monomial of $G$.  Then $A$ induces two kinds of monomials of $H_M$:
either we act on each letter of $A$ with exactly one $g_j'' t^{i_{m+1-j}-a_j}$, or we act on some letter of $A$ more than once and in consequence do not act at all on at least one letter of $A$.
We call the first of these actions {\em permutational} and the second {\em non-permutational}.
The $<$-leading term in $H_M$  always comes from a permutational action.
This is because in a non-permutational action, the smallest power of $t$ is always bounded above by $a_m$ since there will be a missing $h_i t^{b_i}$ which is not acted upon and 
$b_i \leq a_m$ by definition of the monomial order $\prec$;
whereas in a permutational action, the lowest power of $t$ is bounded below by $i_1+b_1-a_m$ which is strictly  bigger than $a_m$ because of our assumption that $i_1 > 2a_m-s$. 

Now, if $\sum b_i < \sum a_i$, then all monomials of
\[\{g_1''t^{i_m-{a_1}},\{g_2'' t^{i_{m-1}-a_2}\cdots, \{g_m''t^{i_1-a_m}, A\}\cdots \}\}\]
have degree $< \deg M$ and so
are strictly $< M$.  We may thus assume that $\sum b_i = \sum a_i$.

A monomial coming from a permutational action upon $A$ looks like
\beq \label{perm-mon} [g_m'', h_{\tau(1)}]t^{i_1+b_{\tau(1)}-{a_m}} [g_{m-1}'',h_{\tau(2)}]t^{i_{2}+b_{\tau(2)}-a_{m-1}}\cdots [g_1'',h_{\tau(m)}]t^{i_m+b_{\tau(m)}-a_1}
\eeq
for some $\tau \in \mf S_m$, where all the Lie brackets are nonzero.  
We claim that
\[ \eqref{perm-mon} \leq M,\]
with equality only when $b_{\tau(i)}= m+1-i$ and $h_{\tau(i)} = g_{m+1-i}$ for all $i$.

Note that $b_{\tau(1)} \leq a_m$ thanks to the definition of $\prec$.
We will have several different cases:
\begin{itemize}
\item  If $b_{\tau(1)}<a_m$
 then our claim is obvious because $i_1+b_{\tau(1)}-a_m < i_1$.  
\item  
If $b_{\tau(1)}=a_m$, then we must have $[g_m'', h_{\tau(1)}] \leq g_1'$, with equality only if $h_{\tau(1)}= g_m$.
This is because, by definition of $\prec$,  $h_{\tau(1)} \leq g_m$ and $<$ has property \eqref{orderB}. 
\item if $b_{\tau(1)} = a_m$ and   $h_{\tau(1)}=g_m$, then 
   we want to show that 
   \[g'_1 t^{i_1} [g_{m-1}'',h_{\tau(2)}]t^{i_{2}+b_{\tau(2)}-a_{m-1}}\cdots [g_1'',h_{\tau(m)}]t^{i_m+b_{\tau(m)}-a_1} \leq g'_1 t^{i_1}\cdots g'_m t^{i_m},\]
   with equality only when $b_{\tau(i)}=a_{m+1-i}$ and $h_{\tau(i)} = g_{m+1-i}$ for all $2 \leq i \leq m$.
   This can  be proved by repeating the same procedure and moving towards the right.
    \end{itemize}
    
    This analysis show that $LT_\prec G$ is the only monomial which contributes to the occurrence of $M$ in $H_M$.
    To finish, we note that by \eqref{magic}, then
    \[ [g_m'', g_m]\dots [g_1'', g_1] = \lambda g_1' \dots g_m'\]
    for some $\lambda \neq 0$.
    Thus $M$ occurs in $H_M$ with nonzero coefficient $C\lambda$, where
    \[ C = \#\{ \tau \in \mf S_m : a_{\tau(i)} = a_{m+1-i}  \mbox{ and } g_{\tau(i)} = g_{m+1-i} \mbox{ for all $i$ } \} \neq 0.\]
    Thus $LT_< H_M = M$.
\end{proof}

To use Lemma~\ref{lem:W-technique-twist} we must prove that we can always find a situation where \eqref{magic} holds.  This is given by the next elementary lemma.
\begin{lemma}\label{lem:zero}
For any $g \in \mc B$ there are $g' \in \mc B_{\Delta}$ and weight-homogeneous $g'' \in \mf g$ so that $[g'', g']$ is a nonzero scalar multiple of $g$.
\end{lemma}
\begin{proof}

If $g = g_{\pm \xi_i, \pm \beta_i} \in \mc B_{\Delta}$, set $g^- = g_{\mp \xi_i, \mp \beta_i}$.
Let $g' = g$ and $g'' = [g, g^-]$.
As $\{ g, g^-, [g, g^-] \} $ form an $\mf{sl}_2$-triple, $[g'', g']$ is a nonzero scalar multiple of $g$.

If $g \in \mf h$ then  we consider the action of $\sigma $ on $g$.  
We give the proof for $r=3$. 
If $|g| = \overline 0$ then $g = h_{\alpha_i}+h_{\sigma(\alpha_i)}+h_{\sigma^2(\alpha_i)}$ for some simple root $\alpha_i$ and we can take $g'=g_{|\alpha_i|, \alpha_i}$ and $g'' = g_{|-\alpha_i|, -\alpha_i}$.
If $|g| = \overline 1$ then
$g = h_{\alpha_i}+\eta h_{\sigma(\alpha_i)} + \eta^2 h_{\sigma(\alpha_i)}$ for some $\alpha_i$.
Set $g' = e_{\alpha_i} + \eta e_{\sigma(\alpha_i)} + \eta^2 e_{\sigma^2(\alpha_i)}$ and $g'' = f_{\alpha_i} + f_{\sigma(\alpha_i)} + f_{\sigma^2(\alpha_i)}$.  
Likewise, if $|g| = \overline 2$ then 
$g = h_{\alpha_i}+\eta^2 h_{\sigma(\alpha_i)} + \eta h_{\sigma(\alpha_i)}$ for some $\alpha_i$.
Set $g' = e_{\alpha_i} + \eta^2 e_{\sigma(\alpha_i)} + \eta e_{\sigma^2(\alpha_i)}$ and $g'' = f_{\alpha_i} + f_{\sigma(\alpha_i)} + f_{\sigma^2(\alpha_i)}$.  
\end{proof}

\subsection{The main reduction result}\label{ssREDUCTION}
We now prove our main reduction result, Proposition~\ref{newcor6}.

\begin{proof}[Proof of Proposition~\ref{newcor6}]
We may assume that $I$ is generated by a single element $F\neq 0$.
For every $\underline{g'} = (g'_1, \dots, g'_m) \in (\mc B_{\Delta})^m$, using Lemma~\ref{lem:arbitrary congruence classes-twist} choose $G_{\underline g'} \in I$ with $LT_\prec G_{\underline{g'}}\equiv \underline{g'}$.
Let $s_{\underline g'}$ be the smallest power of $t$ occurring in $G_{\underline g'}$, and let $S_{\underline g'} $ be the largest power of $t$.
Let $n$ be an integer so that 
\[ n \geq \max (0, \{ 2S_{\underline g'}-s_{\underline g'} : \underline g' \in (\mc B_{\Delta})^m\}).\]
Let $\ell$ be the smallest power of $t$ occurring in $F$.

Now fix $M = g_1 t^{i_1} \dots g_m t^{i_m}$ with $i_1 > n$.  
We construct $H_M$.
Using Lemma~\ref{lem:zero}, choose $g_1', \dots, g_m' \in \mc B_{\Delta}$ and weight-homogeneous $g_1'', \dots, g_m'' \in \mf g$ so that
$[g_j'', g_j']$ is a nonzero multiple of $g_{m+1-j}$ for $1 \leq j \leq m$.
Let $G = G_{(g_1', \cdots, g_m')} \in I$, and write
 $LT_\prec G = g_1' t^{a_1}\dots g_m' t^{a_m}$.
 Let 
 \[ H_M = \{ g_1'' t^{i_m-a_1}, \{ \dots, \{ g_m'' t^{i_1-a_m}, G\}\cdots \} \}.\]
The proof of  Lemma~\ref{lem:W-technique-twist}
shows that
\[ M = LT_< H_M,\]
which is in $I$ as our definition of $n$ ensures that $0 \leq i_1 - a_m \leq i_j - a_{m+1-j}$ for any $j$.

Throughout the proofs of Lemmata~\ref{lem:sl2theta-twist}, \ref{lem:arbitrary congruence classes-twist}, and \ref{lem:W-technique-twist} we acted on $F$ only with positive powers of $t$.  
Thus the smallest $t$-power in $H_M$ is $\geq \ell$.  
\end{proof}

From Proposition~\ref{newcor6} we  deduce our first just-infinite growth result.
\begin{theorem}\label{thm:current}
Let $\mf g$ be a finite-dimensional simple Lie algebra and let $\sigma \in \Aut \mf g$.
\begin{itemize}
    \item[(1)] The enveloping algebra $U(\mf g[t]^\sigma)$ has just-infinite growth.  That is, if $J$ is a nonzero ideal of $U(\mf g[t]^\sigma)$, then $U(\mf g[t]^\sigma)/J$ has polynomial growth.
    \item[(2)] Let $I$ be a nonzero Poisson ideal of $S(\mf g[t]^\sigma)$.  
Then $S(\mf g[t]^\sigma)/I$ has polynomial growth.
\end{itemize}
\end{theorem}

\begin{proof}
If $M$ is a standard monomial in the letters $\mc B[t]^\sigma$,  
we define the {\em modified degree} of $M$ to be
\[ \md M = \deg M + \len M.\]
Let $\mc U^j$ (respectively, $\mc S^j$) be the subspace of $U(\mf g[t]^\sigma)$ (respectively, $S(\mf g[t]^\sigma)$) spanned by standard monomials of modified degree $\leq j$.
As $U(\mf g[t]^\sigma)$ is degree-graded, and, by the PBW theorem, is filtered by length, thus $\mc U^\bullet$ defines a filtration on $U(\mf g[t]^\sigma)$: that is, $\mc U^i \mc U^j \subseteq \mc U^{i+j}$ for all $i, j\in \NN$.
Likewise, $\mc S^\bullet$ defines a filtration on $S(\mf g[t]^\sigma)$.
It suffices to prove that
\[ \dim \frac{\mc U^j + J}{J} \quad \mbox{and} \quad  \dim \frac{\mc S^j + I}{I}\]
have polynomial growth.

Let
\[ \gr_{\md}(U(\mf g[t]^\sigma)) = \bigoplus_{j} \mc U^j/\mc U^{j-1}.\]
This is a (commutative) Poisson algebra: 
if $F \in \mc U^i$ and $G \in \mc U^j$ then $FG-GF \in \mc U^{i+j-1}$ so $\gr_{\md} U(\mf g[t]^\sigma)$ is commutative.

Define
\[ 
\{ \gr_{\md} F, \gr_{\md} G \} = \frac{FG-GF+\mc U^{i+j-2}}{\mc U^{i+j-2}} \in \bigl(\gr_{\md}U(\mf g[t]^\sigma)\bigr)_{i+j-1}.
\]
In fact, as degree alone defines a grading on $U(\mf g[t]^\sigma)$, there is a canonical identification  $\gr_{\md} U(\mf g[t]^\sigma) \cong \gr_{\len}(U(\mf g[t]^\sigma)) = S(\mf g[t]^\sigma)$ as Poisson algebras.
Further, $\gr_{\md} J$ is a nontrivial $\md$-homogeneous Poisson ideal of $S(\mf g[t]^\sigma)$.
Likewise, 
$\gr_{\md} I$ is a nontrivial $\md$-homogeneous Poisson ideal of $S(\mf g[t]^\sigma)$.
As 
\[ \dim \frac{\mc U^j + J}{J} = \dim \frac{\mc S^j + \gr_{\md}J}{\gr_{\md}J}\] 
and similarly for $I$, it suffices to prove:

{\bf Claim:}  Let $K$ be a nontrivial $\md$-homogeneous Poisson ideal of $S(\mf g[t]^\sigma)$. 
Then $\dim \frac{\mc S^j + K}{K}$ is bounded by a polynomial in $j$.  

Let $K' = \gr_{\len}(K)$, which is  a nontrivial $\len$-graded Poisson ideal of $S(\mf g[t]^\sigma)$, and thus meets some $S^m(\mf g[t]^\sigma)$ nontrivially.
Note that $K' \cap S^m(\mf g[t]^\sigma)$ is a $(t \mf g[t])^\sigma$-subrepresentation of $S^m(\mf g[t]^\sigma)$.
By Proposition~\ref{newcor6}, there is $n \in \ZZ_{\geq 1}$ so that 
 for any  standard monomial $M= g_1 t^{i_1} \dots g_m t^{i_m}$  with $i_1 \geq n$, there is $H_M'\in K'$ such that $LT_{<}H_M' =M$.
 Now, for each standard monomial $M$ as above, there is $H_M \in K $ with $H_M' = \gr_{\len} H_M$.
 As the monomial ordering $<$ compares length first, $M = LT_{<} H_M$ as well.
 Further, as $K$ is $\md$-graded we may take $H_M$ to be $\md$-homogeneous.
 
 Let $F \in \mc S^j$.  
Repeatedly using the $H_M$ to reduce monomials involving  $m$ or more powers of $t$ which are bigger than $n$,  we may rewrite $F$ (modulo $K$) without increasing $\md F$ so that no monomial in $F$ contains more than $m-1$ $t$-powers bigger than $n$.
That is, $(\mc S^j + K)/K$ is spanned by the image of the set of standard monomials
$M$ with $\md (M) \leq j$ which admit a factorisation $M=M_1 M_2$, where
$M_1$ is a  standard monomial involving   $t$-powers  $\leq n-1$, and
$M_2$ is a  standard monomial of length $< m$ involving $t$-powers $\geq n$.
Let us call such monomials {\em normal words}, and let
\[ r(j) = \# \{ \text{normal words } M : \md(M) \leq j\} \geq \dim \frac{\mc S^j + K}{K}.\]
It is clear that $r(j) \leq b(j) c(j)$, where
\[ b(j) = \# \{ \mbox{ standard monomials $M_1$ involving only $t$-powers $\leq n-1$ with $\md M_1 \leq j$ }\}\]
and 
\[ c(j) = \# \{ \mbox{ standard monomials $M_2$ of length $< m$ involving only $t$-powers $\geq n$ and with $\md M_2 \leq j$ } \}.\]
Now, modified degree exceeds length, so $b(j) $ does not exceed the number of standard monomials of length $\leq j$ involving $t$-powers between $0$ and $n-1$, which is $ \binom{(\dim \mf g)n+j }{j}$ and is bounded above by a polynomial in $j$ of degree $n(\dim \mf g)$.
And in a monomial $M_2$ of degree at most $j$ and length at most $m-1$ in $t$-powers $\geq n$ there are no more than $j\dim \mf g$ choices for each letter of $M_2$, so $c(j) \leq (j \dim \mf g)^{m-1}$.
Thus $r(j)$ is bounded above by a polynomial in $j$ of degree $n(\dim \mf g) + m-1$.
\end{proof}

\begin{remark}\label{rem:positiveloop}
A small modification to the proof of Theorem~\ref{thm:current} shows that if $J$ is a nonzero ideal of $U((t \mf g[t])^\sigma)$ and $I$ is  a nonzero Poisson ideal of $S((t \mf g[t])^\sigma)$ then $U((t \mf g[t])^\sigma)/J$ and $S((t \mf g[t])^\sigma)/I$ have polynomial growth.  
We leave the details to the reader. 
\end{remark}

\section{Useful technical results}\label{USEFUL}

We will prove the just-infinite growth results Theorems~\ref{ithm:U}(0) and \ref{ithm:S}(0) in the next  section.
In this section we establish a number of useful preparatory results which  will allow similar counting arguments to those in the proof of Theorem~\ref{thm:current} to apply to affine Kac-Moody algebras.

We establish notation, which will apply to the next two sections.  
Let $L$ be an affine Kac-Moody algebra, with central element $c$ and derivation $d$.
Then there are a finite-dimensional simple Lie algebra $\mf g$ and an automorphism $\sigma $ of $\mf g$ so that
\[ L'/(c) \cong \mf g[t, t^{-1}]^\sigma.\]
(Here $L'$ is the derived subalgebra of $L$.)
We may assume that $\sigma$ induces an automorphism of the Dynkin diagram and the root system of $\mf g$; we denote this diagram automorphism by $\sigma$ as well.
Let $r$ be the order of $\sigma$.
Throughout the rest of the paper we fix the meanings of $L$, $\mf g$, $\sigma$, $r$, $d$, $c$ as in this paragraph.
We further fix a primitive $r$th root of unity, $\eta$, and induce a $\ZZ/r\ZZ$-grading on $\mf g$ as in Section~\ref{BASICS}.
We will use other notation from Section~\ref{BASICS} without comment.

Let $\lambda \in \kk$.
Define $U_\lambda(L) = U(L)/(c-\lambda)$ and $S_\lambda(L) = S(L)/(c-\lambda)$.  
As $c-\lambda$ is Poisson central in $S(L)$,  the factor  $S_\lambda(L)$ is a Poisson algebra.
Note that $S_0(L')\cong S(\mf g[t, t^{-1}]^\sigma)$ as Poisson algebras.

We first show that every nonzero Poisson ideal of $S_\lambda(L)$ meets $S_\lambda(L')$.  This requires a technical lemma.

\begin{lemma}\label{lem:centralise}
Let $\lambda \in \kk$.  
Let $G \in S_\lambda(L)$, and suppose that there exists a $\sigma$-equivariant $x \in \mf g$ so that $\{x t^{rj+|x|}, G\} =0$ for all $j \in \ZZ$.  
Then $G \in S_\lambda(L')$.
\end{lemma}
\begin{proof}
Write $G = \sum_{i=0}^n d^i G_i$, where $G_i \in S_\lambda(L')$.
Then for any $\sigma$-equivariant $x \in  \mf g$ and $m \in r\ZZ + |x|$,
\beq\label{starstar} 0 =  \{x t^m, G\} = \sum_{i=0}^n d^i \Bigl(-(i+1)m  xt^m G_{i+1} +  \{ xt^m, G_i\}\Bigr).\eeq
Thus for all $m, i$, the coefficient of $d^i$ in \eqref{starstar} must vanish, and se we have
\beq\label{star3} \{ xt^m, G_i \} = (i+1)m  xt^m G_{i+1}.\eeq
Fix $i$ and let $m \gg 0$ be large enough so that $t^{-m}$ does not occur in $G_i$.
As $m \gg 0$ varies, from the definition of the Poisson bracket on $S_\lambda(L)$ we see that the LHS of \eqref{star3}
changes only in the powers of $t$ which occur, and not in the coefficients.
On the other hand, the expressions for the coefficients of the RHS involve a factor of $m$ and so vary with $m$.
Thus \eqref{star3} cannot hold for all values of $ m \in r \ZZ + |x|$ unless both sides are identically 0.
We conclude that  
 $G_{i} = 0$ for all  $i \geq 1$, and $G = G_0 \in S_\lambda(L')$.
\end{proof}

\begin{corollary}\label{cor:intersection}
Let $\lambda \in \kk$ and let $I$ be a nonzero Poisson ideal of $S_\lambda(L)$.  Then $I \cap S_\lambda(L') \neq (0)$.
\end{corollary}
\begin{proof}
Let $0 \neq G \in I$, and  choose $(\overline s, \alpha) \in \Delta^+$.  There is $k\geq 1$ so that 
\[ \{ g_\alpha t^{ri_k+s}, \dots, \{ g_\alpha t^{ri_1 +s}, G\}\dots\} = 0\]
for all $i_1, \dots, i_k \in \ZZ$; without loss of generality let $k$ be minimal.  
Thus there is some 
\[ H = \{ g_\alpha t^{ri_{k-1}+s}, \dots, \{ g_\alpha t^{ri_{1} +s}, G\} \dots \} \neq 0.\]
By Lemma~\ref{lem:centralise}, $H \in I \cap S_\lambda(L')$.
\end{proof}

We need to extend the orderings $<$ and $\prec$ defined on loop algebras to $U_\lambda(L)$ and $S_\lambda(L)$.
To this end, 
define a $\sigma$-equivariant basis $\mc B$ of $\mf g$ as in Basic Fact (9), and
note that  $\{d \} \cup \{ g t^{r j + |g|} : g \in \mc B, j \in \ZZ\}$ is a basis of $L/(c)$.
We denote this basis by
$\mc B^d[t,t^{-1}]^\sigma$.
We  extend the ordering $<$ on $\mc B[t, t^{-1}]^\sigma$ to an ordering on $\mc B^d[t,t^{-1}]^\sigma$ by saying that $d > g t^n$ iff $n < 0$ and $d< gt^n $ iff $n \geq 0$.

We must modify \eqref{standard} to give a basis of $U_\lambda(L)$ and $S_\lambda(L)$. 
We will say that a {\em standard monomial}  in the elements of $\mc B^d[t, t^{-1}]^\sigma$ is an expression of the form:
\beq\label{d-std} M = g_1 t^{r m_1 + |g_1|} \dots g_i t^{rm_i+ |g_i|} d^j h_1 t^{r n_1 + |h_1|} \dots h_k t^{rn_k+ |g_k|},\eeq
where the $g_a, h_b \in \mc B$ and  $g_1 t^{r m_1 + |g_1|} \leq \dots \leq g_i t^{rm_i+ |g_i|} < d < h_1 t^{r n_1 + |h_1|} \leq \dots \leq h_k t^{rn_k+ |h_k|}$.
We sometimes write a standard monomial as
\[ M = g_1 t^{m_1} \dots g_i t^{m_i} d^j h_1 t^{n_1 }\dots h_k t^{n_k},\]
and when we do so we assume that $\overline{m_a} = |g_a|$ and $\overline{n_a} = |h_a|$ for all $a$: in other words, that this monomial is an element of $S(L)$.

By the PBW theorem, both $U_\lambda(L)$ and $S_\lambda(L)$ have a basis of standard monomials in the elements of $\mc B^d[t, t^{-1}]^\sigma$.
We say that $i+j+k$ is the {\em length} of $M$, which we denote $\len M$, and that the total $t$-power $\sum_a rm_a + |g_a| + \sum_b rn_b + |h_b|$ is the {\em degree} of $M$, which we denote $\deg(M)$.
Note that $U_\lambda(L')$ and $S_\lambda(L')$ have a basis of standard monomials in the elements of $\mc B[t, t^{-1}]^\sigma$.

We extend the monomial orderings $<$ and $\prec$ from Subsection~\ref{ssORDER} to define two orderings $<$ and $\prec$ on standard monomials in the elements of $\mc B^d[t,t^{-1}]^\sigma$:  the ordering $<$ compares length first, then degree, and then compares monomials lexicographically from left to right, whereas $\prec$ compares length first, then degree, then compares monomials lexicographically from right to left.

We use the following reduction lemma, which applies Proposition~\ref{newcor6} to ideals of $U_\lambda(L), U_\lambda(L')$ or Poisson ideals of $S_\lambda(L), S_\lambda(L')$.

\begin{lemma}\label{lem:extendedreduction}
Let $\lambda \in \kk$.
\begin{itemize}
    \item[$(1)$] Let $J$ be a nonzero ideal of $U_\lambda(L)$ or of $U_\lambda(L')$.
    There are $m, \ell, n \in \ZZ$, with $m, n > 0$, so that if
    \[ M= g_1 t^{i_1}\dots g_m t^{i_m}\]
    is a standard monomial in the elements of $\mc B[t, t^{-1}]^\sigma$ with $i_1 \geq n$, then there is $H_M \in J$ so that
    \[ LT_< H_M = M\]
    and so that all $t$-powers occurring in $H_M$ are $\geq \ell$.
\item[$(2)$] Let $I$ be a nonzero Poisson ideal of $S_\lambda(L)$ or of $S_\lambda(L')$.
    There are $m, \ell, n \in \ZZ$, with $m, n > 0$, so that if
    \[ M= g_1 t^{i_1}\dots g_m t^{i_m}\]
    is a standard monomial in the elements of $\mc B[t, t^{-1}]^\sigma$ with $i_1 \geq n$, then there is $G_M \in I$ so that
    \[ LT_< G_M = M\]
    and so that all $t$-powers occurring in $G_M$ are $\geq \ell$.
      \end{itemize}
\end{lemma}

\begin{proof}
$(1)$.  We give the proof for $J \triangleleft U_\lambda(L)$; the proof for $J \triangleleft U_\lambda(L')$ is similar but easier.
We filter $U_\lambda(L)$ by length of monomials, and define a Poisson bracket on $\gr_{\len} U_\lambda(L)$ as in the proof of Theorem~\ref{thm:current}.
The relation 
\[ xt^i yt^j - yt^j xt^i = [x,y] t^{i+j} + i\delta_{i+j, 0}\kappa(x,y)\lambda\]
on $U_\lambda(L)$ induces the Poisson bracket
\[ \{ xt^i, yt^j\} = [x,y] t^{i+j}\]
in $\gr_{\len} U_\lambda(L)$ and so 
$\gr_{\len} U_\lambda(L) $ may be identified with $ S_0(L)$ as Poisson algebras. 
By Corollary~\ref{cor:intersection} applied to $\gr_{\len} J$, which is a nonzero Poisson ideal of $S_0(L)$, there is $0 \neq H \in J$ so that $\gr_{\len} H \in S_0(L')$.

Let $m = \len ( \gr_{\len} H )  = \len H$ and let $\ell$ be the smallest $t$-power occurring in $H$.
Let $I$ be the Poisson ideal of $S_0(L)$ generated by $\gr_{\len} H$.
By Proposition~\ref{newcor6} there is $n \in \ZZ$ so that 
 if $M= g_1 t^{i_1} \dots g_m t^{i_m}$ is a standard monomial with $i_1 \geq n$, there is $G_M \in I \cap S_0(L')$ with $LT_< G_M = M$.
The procedure in the proof of Proposition~\ref{newcor6} that produces $G_M$ (that is, the procedure in Lemmata~\ref{lem:sl2theta-twist}, \ref{lem:arbitrary congruence classes-twist}, and \ref{lem:W-technique-twist}) produces
 $B_1, \dots, B_s \in \mc B[t, t^{-1}]^\sigma$, involving only non-negative $t$-powers, so that
 \[ \{ B_1, \{ \dots, \{ B_s, \gr_{\len} H\} \dots \} = G_M.\]
 
Now, if $B \in \mc B[t,t^{-1}]^\sigma $ and $P \in U_\lambda(L)$, then
\beq\label{eq:gr} 
\gr_{\len} [B, P] = \{ B, \gr_{\len} P\} \quad \mbox{if $\{B, \gr_{\len} P\} \neq 0$.}
\eeq
Let 
\[ H_M = [ B_1, [ \dots, [ B_s,  H] \dots ].\]
Applying \eqref{eq:gr} we see that $G_M = \gr_{\len} H_M$, and as the monomial ordering $<$ compares length first, $LT_< H_M = LT_< G_M = M$.
As the $t$-powers in the $B_i$ are non-negative, the $t$-powers in $H_M$ are no  smaller than $\ell$.

$(2)$.  This proof is similar to the proof of $(1)$.
By Corollary~\ref{cor:intersection} it suffices to give the proof for a Poisson ideal $I $ of $ S_\lambda(L')$.
Let $0 \neq G \in I$.  
Let $m = \len G$ and let $\ell$ be the smallest $t$-power occurring in $G$.

Filter $S_\lambda(L')$ by length, as in the proof of $(1)$, so $\gr_{\len} S_\lambda(L') = S_0(L')$.
We again apply Proposition~\ref{newcor6} to obtain $n$ so that if $M$ is a  standard monomial of length $m$ involving only $t$-powers $\geq n$,
then there is $G_M$ in the Poisson ideal of $S_0(L')$ generated by $\gr_{\len} G$ with $LT_< G_M = M$.

Temporarily, let $\{- , - \}_0$ denote the Poisson bracket in $S_0(L')$ and let $\{-, - \}_\lambda$ denote the Poisson bracket in $S_\lambda(L')$.  
Then, similarly to \eqref{eq:gr}, if $B \in \mc B[t,t^{-1}]^\sigma$ and $P \in S_\lambda(L')$, then
\beq \label{eq:gr2}
\gr_{\len} (\{ B, P\}_\lambda) = \{ B, \gr_{\len} P\}_0 \quad \mbox{if $\{ B, \gr_{\len} P\}_0 \neq 0$.}
\eeq
As in the proof of $(1)$, there are $B_1, \dots, B_s \in \mc B[t,t^{-1}]^\sigma$, involving only non-negative powers of $t$, so that
\[ G_M = \{ B_1, \{ \dots, \{ B_s, \gr_{\len} G\}_0, \dots, \}_0.\]
Let 
\[ H_M = \{ B_1, \{ \dots, \{ B_s, G\}_\lambda, \dots, \}_\lambda.\]
As $I$ is a Poisson ideal, $H_M \in I$.
As before, applying \eqref{eq:gr2} this time, $LT_< H_M = M$, and $H_M$ does not involve any $t$-powers smaller than $\ell$.
\end{proof}
Let $L_+$ be the sub-Lie algebra of $L$ generated by all $g t^n$ with $n > 0$, and similarly define $L_-$ to be generated by all $gt^n$ with $n< 0$.
For all $\lambda$, then $U(L_+), U(L_-)$ are subalgebras of $ U_\lambda(L)$, and similarly $S(L_+), S(L_-) \subseteq S_\lambda(L)$, where these second inclusions are inclusions of Poisson algebras.

We next apply Lemma~\ref{lem:extendedreduction} to show that a nontrivial ideal of $U_\lambda(L)$ must meet $U(L_+)$, and, symmetrically, $U(L_-)$.
We do not know of a way to show this without using growth.

\begin{proposition}\label{prop:Icapcurrent}
Let $\lambda \in \kk$.
\begin{itemize}
    \item[$(1)$]
Let $J$ be a nonzero ideal of $U_\lambda(L)$ or of $U_\lambda(L')$.  There are nonzero elements $H^+ \in J \cap U(L_+)$ and $H^- \in J \cap U(L_-)$.

\item[$(2)$]
Let $I$ be a nonzero Poisson ideal of $S_\lambda(L)$ or of $S_\lambda(L')$.  There are nonzero elements $G^+ \in I \cap S(L_+)$ and $G^- \in I \cap S(L_-)$.
\end{itemize}
\end{proposition}

\begin{proof}
We give the proof if $J \triangleleft U_\lambda(L)$.

By symmetry, it suffices to prove the result for $H_+$.  
Let $A= U_\lambda(L)/J$.
We  claim that $A$ has finite GK-dimension as a right or left $U(L_+)$-module; by symmetry it suffices to consider the GK-dimension as a right module.

The proof of Lemma~\ref{lem:extendedreduction}(1) produces $H \in J$ and, from $H$, integers  $m, \ell, n$ so that
if
    \[ M= g_1 t^{i_1}\dots g_m t^{i_m}\]
    is a standard monomial with $i_1 \geq n$, then there is $H_M \in U(L'_{\geq 0}) H U(L'_{\geq 0})$ so that
    \[ LT_< H_M = M\]
    and so that all $t$-powers occurring in $H_M$ are $\geq \ell$.
    Without loss of generality, $\ell \leq 0$.
    Let 
    \[D= (\dim \mf g)(n-1)+m-1.\]
    We claim that the GK-dimension of $A$ as a $U(L_+)$-module is at most $D$.

Let $V \subset U_\lambda(L)$ be a finite-dimensional   subspace which includes 1; we will show that the  $U(L_+)$-module $X= (U(L_+) V U(L_+) + J)/J$ has $\GK \leq D$.
(As $\GK_{U(L_+)}A$ is by definition the supremum over all finitely generated $U(L_+)$ submodules $A' \subseteq A$ of $\GK_{U(L_+)} A'$, this is sufficient to prove the claim.)
Since $U(L_+)$ is finitely graded, for  any $j \in \ZZ_{\geq 0}$ the subspace $\{x \in X: \deg x \leq j\}$ is finite dimensional,
so it suffices to show that the dimension, considered as a function of $j$, grows as a polynomial of degree $\leq D$. 

We may enlarge $V$ without damage, so assume that $H \in V U(L_+)$.
Let $s$ be  the minimum $t$-power occurring in any element of $V$ and let $m'$ be the maximum length of any element of $V$.
We may enlarge $V$ again so that $s \leq \ell$ (recall that $\ell \leq 0$) and $m' \geq m$ and so that $V = S^{\leq m'}(\mf g_{\overline s} t^s \oplus \mf g_{\overline{s+1}} t^{s+1} \oplus\dots \oplus \mf g_{\overline 0} t^0 \oplus \kk d)$.
The minimal degree of an element of $V$ is $sm' \leq 0$.

Note that if $xt^k \in L_+$ and $M$ is a monomial in $V$, then
\[ xt^k M = M xt^k + \mbox{ a sum of monomials of length $\leq \len(M)$ involving $t$-powers $\geq s$.}\]
Thus $xt^k M \in V U(L_+)$.
By induction, $U(L_+)V U(L_+) = V U(L_+)$. 
Thus $X$ is spanned by standard monomials 
 $M$ which admit a factorization $M=M^0 M^1$, where
 $M^0$ is a standard monomial of length $\leq m'$ involving $d$ and $t$-powers between $s$ and $0$, 
and $M^1$ is a standard monomial involving $t$-powers $\geq 1$.
For all such $M$, our assumption that $s \leq \ell$ and  our choice of $V$ mean that $H_M \in VU(L_+) = U(L_+)V U(L_+)$.
Working modulo $J$ to rewrite the monomials $M^1$ and repeatedly applying Lemma~\ref{lem:extendedreduction}, we see that $X$ is spanned by the image of  standard monomials  $M$ in $ V U(L_+)$ which admit a factorization $M=M^0M_1M_2$, where $M^0$ is a standard monomial of length $\leq m'$ involving $d$ and $t$-powers between $s$ and $0$, 
$M_1$ is a standard monomial involving $t$-powers between $1$ and $n-1$,  and  $M_2$ is a standard monomial of length $<m$ involving only $t$-powers $\geq n$.

Thus, the 
growth of $X$ is thus bounded by the growth of 
\[ Y = \{ M^0 M_1 M_2 \},\]
where $M^0$ is a standard monomial of length $\leq m$ involving $d$ and $t$-powers between $s$ and $0$, 
$M_1$ is a standard monomial involving $t$-powers between $1$ and $n-1$, and
$M_2$ is a standard monomial  involving only $t$-powers $\geq n$ and with $\len M_2  < m$.  
It thus suffices to show that 
 \[ q(j) =  \#\{ M \in Y : \deg(M) \leq j\}\]
is bounded by a polynomial in $j$ of degree $D$.  

Clearly $q(j)$ is bounded by $a b(j) c(j)$, where 
\[  a = \dim S^{\leq m'}(\mf g t^s \oplus \dots \oplus \mf gt^0 \oplus \kk d),\]
 \[ b(j) = 
\#\{ \mbox{ standard monomials $M_1$ of degree $\leq j-sm'$ involving only $t$-powers between $1$ and $n-1$ }\},
\]
and
\begin{multline*}
    c(j) = \\ \# \{  \mbox{ standard monomials $M_2$  of length $< m$  and degree $\leq j-sm'$ involving only $t$-powers $\geq n$ }\}.
\end{multline*} 
As $M_1$ involves only positive $t$-powers, $\deg M_1 \geq \len M_1$, and so $b(j) $ does not exceed the number of standard monomials of length $\leq j-sm'$ involving $t$-powers between $1$ and $n-1$, which is $\leq \binom{(\dim \mf g)(n-1)+j-sm' }{j-sm'}$ and is bounded above by a polynomial in $j$ of degree $(\dim \mf g)(n-1)$.
And in a monomial $M_2$ of degree at most $j-sm'$ and length at most $m-1$ in $t$-powers $\geq n$ there are no more than $(j-sm') \dim \mf g$ choices for each letter of $M_2$, so $c(j) \leq ((j-sm')\dim \mf g)^{m-1}$.
Thus there is $\lambda \in \RR$ so that 
$q(j) \leq \lambda j^D$ for all but finitely many  $j$, proving the claim.

As $\GK U(L_+) = \infty$, the natural map $U(L_+) \to A$  cannot be injective, and thus $U(L_+) \cap J \neq (0)$.
(This argument is essentially Scholium~\ref{scholium}, but for modules, not algebras.)

The proofs of other statements are similar, using other parts of Lemma~\ref{lem:extendedreduction}.
\end{proof}

\section{Just-infinite growth}\label{PROOFS}

We now prove our first main results, Theorems~\ref{ithm:U}(0) and \ref{ithm:S}(0).  Throughout, fix $\lambda \in \kk$.  Let $J$ be a nonzero ideal of $U_\lambda(L)$ and let $A= U_\lambda(L)/J$.
Let $J'$ be a nonzero ideal of $U_\lambda(L')$ and let $A'= U_\lambda(L')/J'$.
Let $I $ be a nonzero Poisson ideal of $S_\lambda(L)$ and let $C= S_\lambda(L)/I$, and let $I '$ be a nonzero Poisson ideal of $S_\lambda(L')$ and let $C'= S_\lambda(L')/I'$.

We first prove Theorem~\ref{ithm:U}(0).
As  $U_\lambda(L)$ is not finitely graded, there are some technical issues in the proof. 
Our solution is to extend the definition of modified degree from the proof of Theorem~\ref{thm:current} so that the associated graded ring $B = \gr_{\md} U_\lambda/\gr_{\md} J$ of $A$ will be (finitely) connected graded.
Here recall that an $\NN$-graded $\kk$-algebra $R= \bigoplus_{n\geq 0} R_n$ is {\em finitely connected graded} if each $\dim R_n < \infty$ and $R_0 = \kk$.
We then  use the Poisson GK-dimension of \cite{PS} to bound the GK-dimension of $A$ and of $A'$.

We extend the earlier definiton of modified degree to define the  {\em modified  degree} of an element of $\mc B^d[t, t^{-1}]^\sigma$ to be
\[\md g t^n  = |n|+1, \quad \md d = 1.\]
Let $S = \gr_{\md} U_\lambda(L)$ and let $S' = \gr_{\md} U_\lambda(L')$.
As $\md([g t^a, h t^b]) \leq  |a + b| + 1< |a| + |b| + 2 = \md(g t^a)+ \md(h t^b)$, 
and $\md ([d, g t^a]) < \md (g t^a) + \md (d)$, thus $U_\lambda(L) $ is almost commutative with respect to the filtration induced by $\md$, and so, as a ring, $S$ is isomorphic to a polynomial ring in the variables $d$ and $g t^{ra+|g|}$, 
graded by $\deg (g t^a) = |a| + 1$ and $\deg d = 1$.  
The subring $S'$ is isomorphic to a polynomial ring in the variables $g t^{ra+|g|}$.
Thus $S$ is connected graded.
Further, $S$ has a Poisson bracket 	
\[ \{ g t^a, h t^b\} = \begin{cases} [g, h] t^{a+b} &  ab\geq 0 \\
   0 & \mbox{else,} \end{cases} \]
   \[ \{ d, g t^a\} = a g t^a\]
   induced from the commutator in $U_\lambda(L)$, as in the proof of Theorem~\ref{thm:current}, and $S'$ is a Poisson subalgebra of $S$.

Let $B = S/\gr_{\md} J$, and note that $B$ is the associated graded ring of $A$ with respect to the filtration induced by $\md$.  
Likewise, let $B'= S'/\gr_{\md} J'$; we have $B' \cong \gr_{\md}(A')$.
As usual, $\gr_{\md} J$ is a Poisson ideal of $S$ (respectively, $\gr_{\md} J' $ is a Poisson ideal of $S'$) 
and so the Poisson bracket on $S$ (respectively, $S'$) 
descends to $B$ (respectively, $B'$).

The point of introducing the filtration $\md$ is that the GK-dimension of $A$ may be computed   from the growth of $B$.
For $j \in \NN$, let $S_{\leq j}$ denote the span in $S$ of standard monomials of modified degree $\leq j$ and similarly define $S'_{\leq j}$.

\begin{proposition}\label{prop:PS}
For $j \in \NN$, let  $B(j)$ be the image of $S_{\leq j}$ 	in $B$ and let $B'(j)$ be the image of $S'_{\leq j}$ in $B'$.
Then 
\[ \GK A = \varlimsup \log_j \dim B(j)\]
and 
\[ \GK A' = \varlimsup \log_j \dim B'(j)\]
\end{proposition}
\begin{proof}
This is \cite[Proposition~3.14]{PS}.  
We must check that the hypotheses of that result apply:  that is,  for $B$, that $\{ B(i), B(j)\} \subseteq B(i+j)$ for all $i,j$, and that for some $s$, we have $B(j) \subseteq B(s)^{\{j\}}$ for all $j$, and similarly for $B'$.
Here recall that if $V$ is a  subspace of $B$, then $V^{\{j\}}$ is defined inductively:
\begin{itemize}
\item $V^{\{0\}} = \kk$
\item For $j \in \NN$, define $V^{\{j+1\}} = V V^{\{j\}} + \{ V, V^{\{j\}}\}$
\item In particular $V^{\{1\}} = V$.
\end{itemize}
The needed properties for $B$ and $B'$ follow immediately from similar properties for $S_{\leq j}$, taking $s=r+1$.
That $\{ B(i), B(j)\} \subseteq B(i+j)$ is immediate.
And by Lemma~\ref{lem:bracket}, 
\beq\label{g0action}
\mf g_{\overline a} t^{(n-1)r+a} = (\ad \mf g_{\overline 0} t^r)^{n-1}(\mf g_{\overline a} t^a) \subseteq (S_{\leq r+1})^{\{n\}} 
\eeq
for $ a \in \{ 0, \dots, r-1\}$.
Thus $S_{\leq j} \subseteq (S_{\leq r-1})^{\{j\}}$ for all $j$.
\end{proof}	

We give the proof of \eqref{g0action}.  

\begin{lemma}\label{lem:bracket}
For $a \in \{0, \dots, r-1\}$, we have
\[ \mf g_{\overline a} = [ \mf g_{\overline 0}, \mf g_{\overline a} ] = \spann([x,y] : x \in \mf g_{\overline 0}, y \in \mf g_{\overline a}).\]
\end{lemma}
\begin{proof}
A basis for $\mf g_{\overline a}$ is made up of the $g_{\overline a, \alpha}$ for $\alpha \in \Delta_{\overline a}$ and the $h^i_{\overline a}$ from Table 1.
We show that all these elements are in $[ \mf g_{\overline 0}, \mf g_{\overline a} ]$.
From  the Basic Facts we have
$[g_{\overline{-a}, -\alpha}, g_{\overline a, \alpha}] \in \mf g_{\overline 0}$, and
$[[g_{\overline{-a}, -\alpha}, g_{\overline a, \alpha}], g_{\overline a, \alpha}]$ is a nonzero scalar multiple of $g_{\overline a, \alpha}$.
Writing
\[ h^i_{\overline a} = \sum_{j=0}^{r'} \eta^{aj} h_{\sigma^j(i)},\]
let
\[ e = \sum_{j=0}^{r'} e_{\sigma^j(\alpha_i)} \in \mf g_{\overline 0}, \quad\quad
f = \sum_{j=0}^{r'} \eta^{aj} f_{\sigma^j(\alpha_i)} \in \mf g_{\overline a}.\]
 Then $[e, f] = h^i_{\overline a}$.
 The lemma follows.
\end{proof}

Let $0 \neq H^+ \in J \cap U(L_+)$ and $0 \neq H^- \in J \cap U(L_-)$ be the elements produced by Proposition~\ref{prop:Icapcurrent}.
Let $G^+ = \gr_{\md} H^+$ and let $G^-= \gr_{\md} H^-$. 
We now apply the  reduction procedure in Proposition~\ref{newcor6} to obtain a reduction result for $B$.  This is:

\begin{lemma}\label{lem:reduction}
There exist positive integers $m$, $n$ so that the following hold:
\begin{itemize}
\item[$(1)$]
Every  standard monomial $M = g_{i_1} t^{j_1}\dots g_{i_m} t^{j_m}$ with $n \leq j_1 \leq \dots \leq j_m$ satisfies
\[  M = H + \sum c_s M_s,\]
where $H \in \gr_{\md} J \cap S(L_+)$ is homogeneous in modified degree, the sum is finite, $c_s \in \kk^\times$,  and the $M_s$ are standard monomials so that for each $t$ we have 
 at least one $t$-power $<n$ featuring in $M_s$.
\item[$(2)$]
Every  standard monomial $M = g_{i_1} t^{j_1}\dots g_{i_m} t^{j_m}$ with $ j_1 \leq \dots \leq j_m \leq -n$ satisfies
\[  M = H + \sum c_s M_s,\]
where $H \in \gr_{\md} J \cap S(L_-)$ is homogeneous in modified degree, the sum is finite, $c_s \in \kk^\times$,  and the $M_s$ are standard monomials so that for each $t$ we have 
 at least one $t$-power $>-n$ featuring in $M_s$.
\end{itemize}
\end{lemma}
\begin{proof}
It suffices to prove $(1)$.  
Noting that $S(L_+) \subset S$ is an inclusion of Poisson algebras, we apply Proposition~\ref{newcor6} to reduce $M$ modulo the Poisson ideal of $S(L_+)$ generated by $G^+$.  
This produces an element $H \in \gr_{\md}  J \cap S(L_+)$ satisfying all claimed properties but $\md$-homogeneity.  Since the Poisson bracket in $S(L_+)$ preserves homogeneity with respect  to modified degree (which equals degree plus length here) and $G^+$ is $\md$-homogeneous, $H$ is $\md$-homogeneous as claimed.  
\end{proof}

We now prove Theorem~\ref{ithm:U}(0).

\begin{proof}[Proof of Theorem~\ref{ithm:U}(0)]
The proofs for $U_\lambda(L)$ and for $U_\lambda(L')$ are very similar; we give the proof for $U_\lambda(L)$.
By Proposition~\ref{prop:PS}, it suffices to show that $\dim B(j)$ has polynomial growth.
(For $U_\lambda(L')$ we show that $\dim B'(j)$ has polynomial growth.)

Let $F \in B(j)$.  
Repeatedly applying Lemma~\ref{lem:reduction} to $F$,  we may rewrite $F$ without increasing $\md F$ so that no monomial in $F$ contains more than $m-1$ $t$-powers bigger than $n$ or more than $m-1$ $t$-powers smaller than $-n$.
That is, $B(j)$ is spanned by the image of the set of standard monomials
$M$ with $\md (M) \leq j$ which admit a factorisation $M=M_0 M_1 M_2$, where
$M_1$ is a standard monomial of length $< m$ involving $t$-powers $\leq -n$,
$M_2$ is a  standard monomial involving  $d$ and $t$-powers between $1-n$ and $n-1$, and
$M_3$ is a  standard monomial of length $< m$ involving $t$-powers $\geq n$.
Let us call such monomials {\em normal words}, and let
\[ r(j) = \# \{ \text{normal words } M : \md(M) \leq j \}.\]
We have seen that $\dim B(j) \leq r(j)$;  an argument very similar to the proof of  Proposition~\ref{prop:Icapcurrent}
shows that $r(j) \leq e(j) c(j)^2$, where
\begin{multline*} e(j) =  \#\{ M : \\
\mbox{ $M$ is a standard monomial involving only $d$ and $t$-powers between $1-n$ and $n-1$ and $\md M \leq j$ } \}
\end{multline*}
and
\begin{multline*} c(j) = \# \{ M_2:\\
\mbox{$M_2$ is a standard monomial  of length $< m$ involving only $t$-powers $\geq n$ with $\deg M_2 \leq j$ }\}.\end{multline*}
We have seen 
that $c(j) \leq (j \dim \mf g)^{m-1}$.
Similarly to the proofs of Theorem~\ref{thm:current} and Lemma~\ref{newcor6}, 
$e(j) \leq \binom{(\dim \mf g) (2n-1)+1+j}{j}$.  
Thus $r(j)$ is bounded by a polynomial in $j$ of degree $(\dim \mf g) (2n-1)+2m-1$. 
\end{proof}

We also have:
\begin{proposition}\label{prop:localisedKM}
Let $L$ be an affine Kac-Moody algebra with central element $c$.  Then $U(L)\otimes_{\kk[c]} \kk(c)$ and $U(L') \otimes_{\kk[c]}\kk(c)$ have  just-infinite growth as  $\kk(c)$-algebras.
\end{proposition}
\begin{proof}
None of the steps  in the proof of Theorem~\ref{ithm:U}(0) used that $\kk$ is algebraically closed.  
Thus we may change the ground field to $\kk(x)$, where $x$ is an indeterminate, to obtain:  for any $\lambda\in \kk(x)$, $U_{\kk(x)}(L)/(c-\lambda)$ has just-infinite growth as a $\kk(x)$-algebra.  
This holds in particular for $\lambda =x$, giving that $U_{\kk(x)}(L)/(c-x) \cong U(L)\otimes_{\kk[c]} \kk(c)$ has just-infinite growth as a $\kk(c)$-algebra, and similarly for $L'$.  
\end{proof}

We will also use modified degree to prove Theorem~\ref{ithm:S}(0).
\begin{proof}[Proof of Theorem~\ref{ithm:S}(0)]
As above, modified degree introduces a filtration on $C$ and on $C'$, and we have:
\[ \gr_{\md} C \cong S/\gr_{\md} I =  R, \quad \gr_{\md}(C') \cong S'/\gr_{\md}(I') = R'.\]
Define $R(j), R'(j)$ to be the elements of $R$ (respectively, $R'$) of modified degree $\leq j$.
As in the proof of Theorem~\ref{ithm:U}, we may use $G^+, G^-$ from Proposition~\ref{prop:Icapcurrent} to find $m, n \in \ZZ$ so that for all $j$, 
$R(j)$ is spanned by the
 image of the set of standard monomials
$M$ with $\md (M) \leq j$ which admit a factorisation $M=M_0 M_1 M_2$, where
$M_0$ is a standard monomial of length $< m$ involving $t$-powers $\leq -n$,
$M_1$ is a  standard monomial involving $d$ and $t$-powers between $1-n$ and $n-1$,
and $M_2$ is a standard monomial of length $< m$ involving $t$-powers $\geq n$.
 Let $D = (\dim \mf g) (2n-1)+2m-1$.  
The same argument as in the proof of Theorem~\ref{ithm:U} shows that 
  $\dim R(j)$ is bounded by a polynomial in $j$ of degree $D$ and so is bounded by some $\lambda j^D$.

Let $V$ be a finite-dimensional subspace of $C$, containing 1, and choose $a$ so that $\gr_{\md} V \subseteq R(a)$.
Then $\dim V^j \leq \dim R(aj) \leq \lambda a^D j^D$.
Therefore $\GK C \leq D$.

The proof that $\GK C' < \infty$ is similar.
\end{proof}

We also have:
\begin{proposition}\label{prop:localisedPoisson}
Let $I$ be a nonzero Poisson ideal of $S(L) \otimes_{\kk[c]}\kk(c)$ and let $I'$ be a nonzero Poisson ideal of $S(L') \otimes_{\kk[c]}\kk(c)$.  Then $S(L) \otimes_{\kk[c]}\kk(c)/I$  and $S(L') \otimes_{\kk[c]}\kk(c)/I'$ have polynomial growth.
\end{proposition}
We leave the proof to the reader.

 \begin{remark}\label{rmk:temp}
Propositions~\ref{prop:localisedKM} and \ref{prop:localisedPoisson} should be viewed as temporary results, as later we will prove, as stated in Theorems~\ref{ithm:U}(2) and \ref{ithm:S}(2), that $U(L)\otimes_{\kk[c]} \kk(c)$ and $U(L')\otimes_{\kk[c]} \kk(c)$ are simple, and that $S(L)\otimes_{\kk[c]} \kk(c)$ and $S(L')\otimes_{\kk[c]} \kk(c)$ are Poisson simple.
\end{remark}

\begin{remark}\label{lem:Vir}
The strategy of proof of Theorems~\ref{ithm:U}(0) and \ref{ithm:S}(0) can be modified to give new proofs of \cite[Theorems~5.3 and 5.6]{IS}.
Recall that the {\em Witt algebra} $W = \kk[t, t^{-1}]\del$ is the Lie algebra of polynomial vector fields on the punctured  line (here $\del = \frac{d}{dt}$) and the {\em Virasoro algebra} $\V$ is the unique nontrivial central extension of $W$.  As a vector space we have
\[ \V = W \oplus \kk c\]
with Lie bracket
\[ [f \del, g \del] = (f g' - f'g) \del + \Res_0(f' g''-g'f'') c, \quad  \mbox{$c$ central.} \]
We describe the necessary modifications to prove \cite[Theorem~5.3]{IS}:  that central quotients of $U(\V)$ have just-infinite growth.

Let $\lambda \in \kk$ and let $J$ be a nonzero ideal of $U_\lambda := U(\V)/(c-\lambda)$.
The proof of Proposition~\ref{prop:Icapcurrent} may be modified (using the reduction in \cite[Lemma~2.2]{IS} instead of Proposition~\ref{newcor6}) to show that there are nonzero $H^+ \in J \cap U(t \kk[t]\del)$ and $H^- \in J \cap U(t^{-1} \kk[t^{-1}]\del)$.
Similarly to our methods here, filter $U_\lambda$ by modified degree, where we define $\md(t^{n+1}\del) = |n|+1$.
Let $G^\pm = \gr_{\md}( H^\pm ) \in \gr_{\md}(J)$.
The reduction argument in \cite[Lemma~2.2]{IS} now gives a version of Lemma~\ref{lem:reduction} for $\gr_{\md}(U_\lambda/J)$ and a similar counting argument to the proof of Theorem~\ref{ithm:U} shows that $U_\lambda/J$ has polynomial growth.
\end{remark}

\section{Simplicity of nontrivial central quotients}\label{SIMPLE}

In this section, we prove Theorem~\ref{ithm:U}(1,2) and Theorem~\ref{ithm:S}(1,2).
Throughout the section, let $L$ be an affine Kac-Moody algebra with central element $c$ and derivation $d$, and let $L'/(c) = \mf g[t, t^{-1}]^\sigma$, where $\mf g$ is a finite-dimensional simple Lie algebra and $\sigma \in \Aut(\mf g)$.

One of our main techniques will be to use the following corollary of Theorem~\ref{ithm:U}(0).
(This was the reason for proving such a general version of this theorem, even though the result for $\lambda \neq 0$ will soon be superseded.)
\begin{proposition}\label{prop:scholium}
Let $\lambda \in \kk^*$ and let $B$ be either $U_\lambda(L)$ or $U_\lambda(L')$.  Let $A $ be a $\kk$-subalgebra of $B$ with $\GK A = \infty$.
If $J$ is a nonzero ideal of $B$ then $J\cap A \neq (0)$.
\end{proposition}
\begin{proof}
Combine Scholium~\ref{scholium} with Theorem~\ref{ithm:U}(0).
\end{proof}

We will show that we can restrict without loss of generality to the case $\mf g = \mf{sl}_2$ (and $\sigma = 1$).
Thus to begin we consider this case.

\subsection{The $\mf{sl}_2$ case}\label{sssl2}

In this subsection assume now that $L$ has type $A^{(1)}_1$ (so $\mf g = \mf{sl}_2$ and $\sigma = 1$).
The derived subalgebra $L'$ of $L$  is isomorphic as a vector space to $\mf{sl}_2[t, t^{-1}] \oplus \kk c$.  
We fix notation for elements of  $ L'$:    let $e, f, h$ be the standard basis of $\mf{sl}_2$ and let $e_n = e  t^n$, $f_n = f  t^n$, $h_n = h  t^n$.  
Let $\lambda \in \kk^*$.
In this subsection we will show  that $U_\lambda(L')$ is simple.

Let $\kk[e_\bullet]$ denote the polynomial ring $ \kk[e_n : n \in \ZZ]$, which is  a subalgebra of $U_\lambda(L')$.
Likewise, let $\kk[h_\bullet] =  \kk[h_n: n \in \ZZ]$; note that this is not isomorphic to a subalgebra of  $U_\lambda(L')$, as $h_n, h_{-n}$ do not commute in $U_\lambda(L')$ unless $n =0$.
Given $n \in \ZZ$, let $\kk[h_{\geq n}] = \kk[h_n, h_{n+1}, \dots]$ and similarly define $\kk[e_{\geq n}]$.

We begin by considering the structure of an ideal $J' $ of $\kk[e_\bullet]$ which is of the form $J \cap \kk[e_\bullet]$ for some ideal $J$ of $U_\lambda(L')$.  By Proposition~\ref{prop:scholium}, if $J$ is nonzero, then so is $J'$.

We first  show that $\kk[h_\bullet]$ acts on $\kk[e_\bullet]$, and study this action; this is relevant as any $J'$ as above is an $\kk[h_\bullet]$-submodule of $\kk[e_\bullet]$.
Given $i \in \ZZ$, we write $\frac{\del}{\del i}$ for the derivation $\frac{\del}{\del e_i}$ of $\kk[e_\bullet]$, and similarly, given $I = (i_1, \dots, i_k) \in \ZZ^k$ define $\frac{\del^k}{\del I} = \frac{\del^k}{\del i_1 \dots \del i_k}$.  
Note that the operator $\frac{\del^k}{\del I}$ depends only on the orbit of $I$ under the symmetric group $\mf S_k$.

Going forward, we  without comment use the notation that if $I \in \ZZ^k$, then $I = (i_1, \dots, i_k)$.

\begin{lemma}\label{lem:adh}
\begin{itemize}
\item[(1)]
The subalgebra $\kk[e_\bullet]$ of $U_\lambda(L')$ is preserved by the adjoint action of the $h_n$, and this action gives $\kk[e_\bullet]$ the structure of a $\kk[h_\bullet]$-module.
\item[(2)] Specifically, $h_n$ acts on $\kk[e_\bullet]$ by
\beq\label{ad h}
\ad h_n = 2  \sum_{i \in \ZZ} e_{n+i} \frac{\del}{\del i}.
\eeq
\item[(3)]
For all $\vec{p} = (p_1, \dots, p_k) \in \ZZ^k$, there is $\theta_{\vec p} \in \kk[h_\bullet]$  so that $\theta_{\vec p}$ acts on $\kk[e_\bullet]$ by
\[  \sum_{I \in \ZZ^k} e_{i_1+p_1} e_{i_2+p_2} \dots e_{i_k +p_k} \frac{\del^k}{\del I}.\]
Further, if $\vec p \in \NN^k$ then we may take $\theta_{\vec p} \in \kk[h_0, h_1, \dots]$.
\item[(4)]  If $J'$ is the restriction of an ideal of $U_\lambda(L)$ or of $U_\lambda(L')$ to $\kk[e_\bullet]$, then $J'$ is a $\kk[h_\bullet]$-submodule of $\kk[e_\bullet]$.
\end{itemize}
\end{lemma}
\begin{proof}
$(2)$ is immediate from the fact that $[h_n, e_i] =2 e_{n+i}$ and that taking the commutator with an element is a derivation.
This shows that the adjoint action of the $h_n$ preserves $\kk[e_\bullet]$.

Let $H$ be the   Lie subalgebra of $L'$ generated by the $h_n, n \in \ZZ$.
(Note that $c \in H$; in fact, it is easy to see that $H$ is isomorphic to the infinite-dimensional Heisenberg algebra.) 
We thus have an adjoint $H$-action and induced $U(H)$-action on $\kk[e_\bullet]$; but the adjoint action of $c$ is trivial, so this descends to an action of  $U(H)/(c) \cong \kk[h_\bullet]$.
This proves $(1)$.
Further, an ideal  of a ring is closed under commutation with any element, establishing $(4)$.

It remains to prove $(3)$.
The case $k=1$ is given by \eqref{ad h}:  set $\theta_n = h_n/2$.  
Now let $\vec{r} = (p_1, \dots, p_k, q) \in \ZZ^{k+1}$.
Let $\vec{p} = (p_1, \dots, p_k)$ and denote the standard basis of $\ZZ^k$ by $\{ 1_1 = (1, 0, \dots, 0), 1_2, \dots, 1_k\}$.
Then for $F \in \kk[e_\bullet]$  we have 
\begin{align*}
\theta_q \theta_{\vec p}(F)  & = \sum_{j \in \ZZ} e_{j+q} \frac{\del}{\del j} \Bigl( \sum_{I \in \ZZ^k} e_{i_1+p_1} \dots e_{i_k +p_k} \frac{\del^k F}{\del I} \Bigr) \\ 
&  =  \sum_{j \in \ZZ, I \in \ZZ^k}  e_{j+q} \Bigl(e_{i_1+p_1} \dots e_{i_k +p_k}  \frac{\del^{k+1} F}{\del j \del I} + 
 \frac{\del^k F}{\del I} \frac{\del}{\del j}(e_{i_1+p_1} \dots e_{i_k +p_k}) \Bigr) \\
& =  \theta_{\vec r}(F) + 
 \sum_{I \in \ZZ^k} \frac{\del^k F}{\del I} \sum_{\ell=1}^k \sum_{j \in \ZZ} e_{j+q}\delta_{j, i_\ell+p_\ell} (e_{i_1+p_1} \dots \widehat{e_{i_{\ell}+p_{\ell}}} \dots e_{i_k + p_k}),
\end{align*}
where the notation $\widehat{e_{i_{\ell}+p_{\ell}}}$ means to omit this term.
But 
\[ \sum_{j \in \ZZ} e_{j+q}\delta_{j, i_\ell+p_\ell} (e_{i_1+p_1} \dots \widehat{e_{i_{\ell}+p_{\ell}}} \dots e_{i_k + p_k}) = 
e_{i_\ell+p_\ell+q} e_{i_1+p_1} \dots \widehat{e_{i_{\ell}+p_{\ell}}} \dots e_{i_k + p_k} ,\]
and so
\[ \theta_q \theta_{\vec p} = \theta_{\vec r} + \sum_{\ell=1}^k \theta_{\vec p + q 1_\ell}.\]
By induction on $k$ we have that  $\theta_{\vec r} \in \kk[h_\bullet]$, as needed.
\end{proof}

Lemma~\ref{lem:adh} gives strong constraints on the structure of $\kk[h_\bullet]$-subrepresentations of $\kk[e_\bullet]$.
To see some of the implications, we will use the length grading on $\kk[e_\bullet]$, given by setting $\deg e_i =1$ for all $i$.
For $k \in \NN$, let $\kk[e_\bullet]_k$ denote the vector space of homogeneous elements of $\kk[e_\bullet]$ of length $k$.

\begin{corollary}\label{cor:translate}
Let $J'$ be any $\kk[h_\bullet]$-subrepresentation of $\kk[e_\bullet]$, where the action of $\kk[h_\bullet]$ is given by  Lemma~\ref{lem:adh}.  
Then:
\begin{itemize}
\item[(1)] $J'$ is length-graded;
\item[(2)] $J'$ is closed under the translation automorphism
\[ T:  e_n \mapsto e_{n+1}\]
of $\kk[e_\bullet]$, as well as under $T^{-1}$.
\end{itemize}
\end{corollary}

Before proving Corollary~\ref{cor:translate}, we give an elementary result on the $\theta$-operators of Lemma~\ref{lem:adh}.

\begin{sublemma}\label{sublem:foo}
Let $T_k= \frac{1}{k!}\theta_{1^k} \in \kk[h_\bullet]$, where $1^k = (\underbrace{1, \dots, 1}_k)$. 
Then the adjoint action of $T_k$ applies the translation operator $T$ to all elements of $S[e_\bullet]_k$.  
\end{sublemma}
\begin{proof}
For  $I \in \ZZ^k$,  let $e_I = e_{i_1} \dots e_{i_k}$.
Then 
\[T^{-1}(\theta_{1^k}( e_I))  = \sum_{J \in \ZZ^k} e_J \frac{\del^k e_I}{\del J} 
=\sum_{J \in \mf S_k\cdot I} e_J \frac{\del^k e_I}{\del J} 
 = | \mf S_k \cdot I |  e_I \frac{\del^k e_I}{\del I},
 \]
as $e_I$ and $\frac{\del^k}{\del I}$ depend only on the $\mf S_k$-orbit of $I$.
But $\frac{\del^k }{\del I}e_I = | \Stab_{\mf S_k}(I)|$ so $T^{-1}(\theta_{1^k} (e_I)) = k! e_I$ by the orbit-stabilizer theorem.
\end{proof}
 
\begin{proof}[Proof of Corollary~\ref{cor:translate}]
$(1)$.  A nonzero element of $\kk[e_\bullet]_k$ is an eigenvector for the adjoint action of $\theta_0 = h_0/2$ with eigenvalue $k$.  
As $J'$ is closed under $\theta_0$ it decomposes as a sum of $\theta_0$-eigenspaces, which gives the grading.

For $(2)$, let $E \in J'$; by $(1)$ we may assume that $E$ is homogeneous of length $k$.  
Define $T_k $ as in Sublemma~\ref{sublem:foo}.  
Likewise, let $T^{-1}_k = \frac{1}{k!}\theta_{(-1)^k}$.  
By the sublemma,  $(\ad T_k)(E) = T(E)$ and likewise, $(\ad T^{-1}_k)(E) = T^{-1}(E)$.  
Thus these are are in $J'$, as needed. 
\end{proof}

It is useful to observe that the translation $T$ extends to an automorphism of the loop algebra of $\mf{sl}_2$.

\begin{lemma}\label{lem:translate}
The linear map $T:  \mf{sl}_2[t, t^{-1}] \to \mf{sl}_2[t, t^{-1}]$ defined by
\[ T(e_n) = e_{n+1}, \quad T(h_n) = h_n, \quad T(f_n) = f_{n-1}\]
is a Lie algebra automorphism of $\mf{sl}_2[t, t^{-1}]$. \qed
\end{lemma}

We will need two more computations in $U( L')$.
\begin{lemma}\label{lem:comp1}
The following statements hold in $U( L')$.
\begin{itemize} 
\item[(1)]
We have:
\[ [f_{-a}, e_a^m ] = -m e_a^{m-1}(h_0 + ac)-2 \binom{m}{2} e_a^{m-1}\]
for all $a \in \ZZ, m \in \NN$.
\item[(2)]
For any $E \in \kk[e_{\geq 0}]_k $ we have
\[ [f_0, E] \in \kk[e_{\geq 0}]_{k-1} \cdot (\kk \cdot (h_{\geq 0}) \oplus \kk).\]
\end{itemize}
\end{lemma}
\begin{proof}
$(1)$.  The statement is immediate for $m=1$.  
(Note that our convention is that $[f_{-a}, e_a] = -h_0 - ac$ in $L'$.)
Then by induction we have
\begin{align*}  [f_{-a}, e_a^{m+1}] &= [f_{-a}, e_a^m] e_a + e_a^m [f_{-a}, e_a] \\
&= \bigl(- m e_a^{m-1}(h_0 + ac)-2 \binom{m}{2} e_a^{m-1} \bigr) e_a - e_a^m (h_0 + ac ) \\
&= -m e_a^{m-1} \bigl( e_a (h_0 + ac) + 2e_a \bigr) -2 \binom{m}{2} e_a^m - e_a^m (h_0 + ac) \\
& = -(m+1) e_a^{m}(h_0 + ac)-2 \binom{m+1}{2} e_a^{m},
 \end{align*}
 proving the statement.
 
 $(2)$ is proved by a similar induction.
 For $k=1$ it is immediate, since $[f_0, e_a] = -h_a$ for any $a \geq 0$.
 
 Now let $M$ be a length $k$ monomial in $e_{\geq 0}$ and let $a \geq 0$.
 By induction, write $[f_0, M] = \sum_{i \geq 0} N_i h_i + N$ where $N_i, N \in \kk[e_{\geq 0}]_{k-1}$.
 Then
 \[ [f_0, M e_a] = [f_0, M] e_a + M [f_0, e_a]  = (\sum N_i h_i + N) e_a - M h_a = \sum( N_i e_a h_i + 2 N_i e_{a+i}) - M h_a.\]
 This is in $\kk[e_{\geq 0}]_k \cdot (\kk \cdot (h_{\geq 0}) \oplus \kk)$, as needed.
 \end{proof}

Let $B$ be the subalgbra of $U(L')$ generated by the $e_n$ and the $h_n$ (so $c \in B$).  
Observe that if we set $T(c) = c$ then the formulae in Lemma~\ref{lem:translate} also give an automorphism of $B$.
Further, if $E \in \kk[e_\bullet] \subset  B$, then, by an argument similar to the proof of Lemma~\ref{lem:comp1}(2), $[f_a, E] \in B$ for any $a$.

\begin{lemma}\label{lem:comp2}
Let $E \in \kk[e_{\geq 0}]$, which we think of as a subalgebra of $U(L')$.
Write $E = e_0^m E' + E''$, where $E', E'' \in \kk[e_{\geq 1}]$.
Then
\[ T([f_0, E])- [f_{-1}, T(E)] = c m e_1^{m-1} T(E').\]
\end{lemma}
\begin{proof}
By Lemma~\ref{lem:comp1} we have
\begin{multline} \label{star} T([f_0, E])-[f_{-1}, T(E)] = \\
T \bigl( (-m e_0^{m-1} h_0 - 2 \binom{m}{2}e_0^{m-1}) E' + e_0^m [f_0, E'] + [f_0, E'']\bigl)
 - \\
 (-m e_1^{m-1} (h_0 + c)-2 \binom{m}{2} e_1^{m-1})T(E')- e_1^m[f_{-1}, T(E')]- [f_{-1}, T(E'')]\\
 = c m e_1^{m-1} T(E') + T(e_0^m[f_0, E'])-e_1^m[f_{-1},T(E')] + T([f_0, E'']) - [f_{-1}, T(E'')].
 \end{multline}
 Now, if $n \geq 0$ then $[f_0, e_n] = -h_n$ in $U(L')$.
 We thus observe  that, as $E', E'' \in \kk[e_{\geq 1}]$, the commutators $[f_0, E']$ and $[f_0, E'']$ agree with the commutators in $U_0(L') = U(\mf{sl}_2[t, t^{-1}])$, and likewise for the commutators $[f_{-1}, T(E')]$ and $[f_{-1}, T(E'')]$.  
 It follows that, as $T$ is an automorphism of $\mf{sl}_2[t, t^{-1}]$ and thus extends to an automorphism of $U(\mf{sl}_2[t, t^{-1}])$, we have
 \[ T(e_0^m [f_0, E']) = e_1^m[f_{-1}, T(E')] \quad \mbox{and} \quad T([f_0, E'']) = [f_{-1}, T(E'')] .\]
 Thus \eqref{star} simplifies to 
 \[ cm e_1^{m-1} T(E'),\]
 as claimed.
\end{proof}

We now prove a version of Theorem~\ref{ithm:U}(1,2) for $\mf g = \mf{sl}_2$.

\begin{theorem}\label{thm:sl2}
Let $L$ be an affine Lie algebra of type $A_1^{(1)}$, so $L \cong \widehat{\mf{sl}_2}$.
\begin{itemize}
\item[(1)]
For any $\lambda \in \kk^*$, the algebra $U_\lambda(L')$ is simple.
\item[(2)]
Any nonzero ideal of $U(L')$ contains a nonzero polynomial in $c$; equivalently, $U(L')\otimes_{\kk[c]} \kk(c)$ is simple.
\end{itemize}
\end{theorem}
\begin{proof}
$(1)$.  Let $J$ be a nonzero ideal of $U_\lambda(L')$.
By Proposition~\ref{prop:scholium}, $J' = J \cap \kk[e_\bullet] \neq (0)$.
 By Lemma~\ref{lem:translate}(4), $J'$ is a $\kk[h_\bullet]$-subrepresentation (under the adjoint action) of $\kk[e_\bullet]$.
Thus, $J'$ is length-graded, by Corollary~\ref{cor:translate}(1).
Let $E$ be a nonzero homogeneous element of $J'$ of minimal length, say  $k$.  
We show that we must have  $k =0$; note that if $k=0$ then  $1 \in J$.

Suppose that $k >0$.  
Let $e_m$ be the smallest letter in $E$:  that is, $m$ is the minimum $i$ so that $e_i$ appears in $E$.
Let $E(i) := T^{i-m}(E)$, which is a translate of $E$ with smallest letter $e_i$.  
By Corollary~\ref{cor:translate}(2), the $E(i)$ are in $J'$.
Thus $[f_0, E(0)], [f^{-1}, E(1)] \in J$.
By Lemma~\ref{lem:comp1}
\[ [f_0, E(0)] \in J \cap (\kk \cdot (h_{\geq 0}) \oplus \kk) \cdot \kk[e_\bullet]_{k-1}.\]

Now, consider the element $T_{k-1} $ defined in the proof of Corollary~\ref{cor:translate}(2), which lies in $\kk[h_{\geq 0}]$, considered as a (commutative) subalgebra of $U_\lambda(L')$.
As $\kk[h_{\geq 0}]$ is commutative, the adjoint action of $T_{k-1}$ on $[f_0, E(0)]$ does not affect the $h_i$ and so by the proof of Corollary~\ref{cor:translate}(2) it applies the translation operator $T$.
Thus 
\[ (\ad T_{k-1})([f_0, E_0]) = T([f_0, E_0]).\]
This element is in $J$, which is closed under the adjoint action of $\kk[h_\bullet]$.

Write $E(0) = e_0^m E' + E''$ for some $m \in \NN$ and $E', E'' \in \kk[e_{\geq 1}]$; as $e_0$ appears in $E(0)$ we have $m \geq 1$ and $E' \neq 0$.
Lemma~\ref{lem:comp2} shows (as $\lambda \neq 0$) that $e_1^{m-1}T(E') \in J'$.
But this is a nonzero homogeneous element of length $k-1$,  contradicting the minimality of $k$.

$(2)$.  Suppose that $J$ is a nonzero ideal of $U(L')$ with $J \cap k[c] = (0)$.
Then $J(c) := J \otimes_{\kk[c]} \kk(c)$ is a nonzero ideal of $U(L')  \otimes_{\kk[c]} \kk(c)$.
Combining Scholium~\ref{scholium} and Proposition~\ref{prop:localisedKM} we see that $J(c) \cap \kk(c)[e_\bullet] \neq (0)$ and clearing denominators we obtain a nonzero $E \in J \cap \kk[c, e_\bullet]$.

As in the proof of $(1)$ we can apply Lemma~\ref{lem:comp2} to reduce the length of $E$.  Thus we reduce to the case that $E$ is homogeneous of length $0$ in the $e_i$:  that is, $E \in \kk[c]$, giving a contradiction.
\end{proof}

\subsection{The general case}\label{ssgeneral}
We now let $L$ be an arbitrary affine Lie algebra.
To complete the proof of Theorem~\ref{ithm:U}, we note that $L$ contains (in fact, many choices of) a subalgebra isomorphic to $\widehat{\mf{sl}_2}$.

\begin{lemma}\label{lem:affinesl2subalgebra}
Let $L$ be an affine Kac-Moody algebra, with Chevalley generators $e_i, f_i, h_i$, central element $c$, and derivation $d$.
For any $f_i$, there is a  subalgebra $\overline L$ of $L$ which contains $f_i$, $c$ and $d$ and is isomorphic to $\widehat{\mf{sl}_2}$ via an isomorphism which sends the positive Borel of $\widehat{\mf{sl}_2}$ into the positive Borel of $L$.
\end{lemma}

This is proved in \cite{CRRRV}, but we give the computation explicitly here.
\begin{proof}
There is a simple Lie algebra $\mf g$ and an automorphism $\sigma $ of $\mf g$ so that
$ L'/(c) \cong \mf g[t, t^{-1}]^\sigma.$
Let $r$ be the order of $\sigma$, so $r \in \{1, 2, 3\}$.
There are $e, f \in \mf g$, forming part of an $\mf{sl}_2$-triple, and $k \in \ZZ$ so that $e_i = et^k, f_i = ft^{-k}$.  
Let $h = [e, f]$, and let
\[ \overline L = \kk \cdot ( \{ e_i t^{rn} \}_{n \in\ZZ}, \{ f_i t^{rn} \}_{n \in \ZZ} , \{ h t^{rn} \}_{n \in\ZZ}, c, d).\]
From the definition of the Lie bracket in $L$ we have  
\[ [e_i t^{rn}, f_i t^{rm}] = h t^{r(n+m)} +(k+rn) \delta_{n,-m}\kappa(e,f) c,\]
where $\kappa$ is the Killing form on $\mf g$, so $\overline L$ is a Lie subalgebra of $L$. 
It is not hard to see that $\overline L'/(c) \cong \mf{sl}_2[t, t^{-1}]$ via an isomorphism which sends the positive Borel to the positive Borel, so $\overline L'$ is a central extension of $\mf{sl}_2[t, t^{-1}]$.
To show that  $\overline L \cong \widehat{\mf{sl}_2}$ it is enough to show this central extension is nontrivial; in other words, that $\kappa(e,f) \neq 0$. 

  If $r =1$ then $e, f$ span root spaces of $\mf g$.  
 By \cite[Proposition~8.1]{Hum}, $e$ is $\kappa$-orthogonal to all root spaces of $\mf g$ except for $\kk f$.  
 As $\kappa$ is nondegenerate, $\kappa(e,f) \neq 0$.

If $r>1$, then from the tables in \cite[pp. 128-9]{Kac} we see that there are root vectors $e', f' \in \mf g$ so that $\{e', f', [e',f']\}$ form an $\mf{sl}_2$-triple, and an $r$'th root of unity $\eta$ so that 
\[ e = \sum_{i=0}^{r-1} \eta^i \sigma^i(e'), \quad\quad f = \sum_{i=0}^{r-1} \eta^{r-i} \sigma^i(f').\]
Applying \cite[Proposition~8.1]{Hum} again we have
\[ \kappa(e,f) = \sum_{i=0}^{r-1} \kappa(\sigma^i(e'), \sigma^i(f')) = r \kappa(e', f') \neq 0,\] 
where we have used that $\sigma$ preserves the Killing form $\kappa$.  
 Thus $\overline L'$ is a nontrivial central extension of $\mf{sl}_2[t, t^{-1}]$ and thus $\overline L \cong \widehat{\mf{sl}_2}$.
\end{proof}

\begin{remark}\label{rem:c}
Write the standard basis of $\widehat{\mf{sl}_2}$ as $\{ et^n\} \cup \{ f t^n\} \cup \{h t^n\} \cup \{ C, d\}$, where $C$ is central.
We caution that the  isomorphism $\overline L \to \widehat{\mf{sl}_2}$ constructed above may not send $c$ to $C$; it does, however, send $c$ to a nonzero scalar multiple of $C$.
\end{remark}

We now complete the proof of Theorem~\ref{ithm:U} by proving parts $(1)$ and $(2)$, which are, respectively, parts (1) and (2) of the next result.
\begin{theorem}\label{thm:U12}
Let $L$ be an affine Lie algebra.
\begin{itemize}
\item[(1)] For any $\lambda \in \kk^*$, the algebras $U_\lambda(L)$ and $U_\lambda(L')$ are simple.
\item[(2)] Any nonzero ideal of $U(L)$ or of $U(L')$ contains a nonzero element of $\kk[c]$; equivalently, the algebras $U(L) \otimes_{\kk[c]}\kk(c)$ and $U(L') \otimes_{\kk[c]}\kk(c)$ are simple.
\end{itemize}
\end{theorem}
\begin{proof}
By Lemma~\ref{lem:affinesl2subalgebra}, let $\overline L$ be a Lie subalgebra of $L$ which is isomorphic to $\widehat{\mf{sl}_2}$.

$(1)$.   Let $J$ be a nonzero ideal of $U_\lambda(L)$ or of $U_\lambda(L')$.
By Proposition~\ref{prop:scholium}, $J \cap U_\lambda(\overline L') \neq (0)$ and thus $1 \in J$ by Theorem~\ref{thm:sl2}(1).

$(2)$.  This proof is similar, applying Scholium~\ref{scholium}, Proposition~\ref{prop:localisedKM}, and Theorem~\ref{thm:sl2}(2)  to the localised ideal $J \otimes_{\kk[c]}\kk(c)$.
\end{proof}

We next finish the proof of Theorem~\ref{ithm:S} by proving parts (1) and (2).

\begin{theorem}\label{thm:S12}
Let $L$ be an affine Lie algebra.
\begin{itemize}
\item[(1)] For any $\lambda \in \kk^*$, the algebras $S_\lambda(L)$ and $S_\lambda(L')$ are Poisson simple.
\item[(2)] Any nonzero Poisson ideal of $S(L)$ or of $S(L')$ contains a nonzero element of $\kk[c]$; equivalently, $S(L) \otimes_{\kk[c]}\kk(c)$ and $S(L') \otimes_{\kk[c]}\kk(c)$ are Poisson simple.
\end{itemize}
\end{theorem}

\begin{proof}
Again, let $\overline L$ be a Lie subalgebra of $L$ which is isomorphic to $\widehat{\mf{sl}_2}$.
Define $e_n, h_n, f_n \in \overline L'$ as in Subsection~\ref{sssl2}.

$(1)$.
Let $J$ be a nonzero Poisson ideal of $S_\lambda(L)$ or $S_\lambda(L')$.
By Theorem~\ref{ithm:S}(0) and Scholium~\ref{scholium},  $J' := J \cap \kk[e_\bullet] \neq (0)$.

Since $J$ is closed under Poisson brackets with any element of $\overline L'$, again $J'$ is a $\kk[h_\bullet]$-subrepresentation of $\kk[e_\bullet]$.
Thus Corollary~\ref{cor:translate} applies to $J'$, and again $J'$ is graded and invariant under the translation operator $T$.
Let $E \in J'$ be a nonzero homogeneous element of minimal length $k$; again we will show that $k=0$ and so  $1 \in J $.

Suppose that $k>0$. Again let $e_m$ be the smallest letter in $E$ and let
$E(i) := T^{i-m}(E)$ for $i\in \ZZ$, so $E(i) \in J'$ has smallest letter $e_i$.

Writing $E(0) =  e_0^m E' + E''$ for $E', E'' \in \kk[e_{\geq 1}]$ we have
\begin{multline*}
 T(\{ f_0, E(0)\})- \{ f_{-1}, E(1)\} \\
 = T\bigl(-m e_0^{m-1}h_0 E' + e_0^m \{f_0, E'\} + \{ f_0, E''\}\bigr) 
 - \bigl(-m e_1^{m-1}(h_0+ \lambda) T(E')+ e_1^m \{ f_{-1}, T(E')\}+ \{ f_{-1}, T(E'')\}\bigr) \\
 = \lambda m e_1^{m-1}T(E').
 \end{multline*}
This is  a nonzero homogeneous element of $J'$ of length $k-1$, giving a contradiction.

The proof of $(2)$ combines a similar calculation, Proposition~\ref{prop:localisedPoisson}, and a modification of the proof of Theorem~\ref{thm:U12}(2).  
We leave the details to the reader.
 \end{proof}
 
\section{Other applications}\label{APPLICATIONS}

In this section, we give other applications of our growth results.

In Proposition~\ref{prop:Icapcurrent} we used growth to show that nontrivial ideals of $U_\lambda(L)$ must meet $U(L_+)$.
However, we can use Theorem~\ref{ithm:U}(0) to obtain much stronger results of a similar flavour.

For simplicity, we state this next result only for untwisted loop algebras, although a similar result clearly holds in the twisted case.
Corollary~\ref{cor:cute}, which is an immediate consequence, seems rather surprising without the growth context.
\begin{proposition}\label{prop:cute}
 Let $ \mf g$ be a finite-dimensional simple Lie algebra.  Let $J$ be a nonzero ideal of $U(\mf g[t, t^{-1}])$ and let $\mf k$ be any infinite-dimensional Lie subalgebra of $\mf g[t, t^{-1}]$.  Then $J \cap U(\mf k)$ is nonzero.
\end{proposition}
\begin{proof}
This is a direct application of Proposition~\ref{prop:scholium} with $\lambda =0$.
\end{proof}

\begin{corollary}\label{cor:cute}
In the situation of Proposition~\ref{prop:cute},  if $X \subset \ZZ$ is any infinite set, for example $X$ consists of all of the powers of 7 or all of the primes, and $g$ is any element of $\mf g$, then $J$ contains an element involving only $g t^x$ for $x \in X$.  
\end{corollary}

\begin{proof}
The vector space spanned by $\{ g t^x : x\in X\}$ is an infinite-dimensional (abelian) Lie subalgebra of $\mf g[t, t^{-1}]$, so this follows directly from Proposition~\ref{prop:cute}.
\end{proof}

It is well-known that the enveloping algebras $U(\mf g[t]^\sigma)$, $U(L)$, etc., are not left or right noetherian.
However, the results of this paper, as well as \cite[Corollary~5.14]{LSS}, raise the natural question of whether they satisfy the ascending chain condition on two-sided ideals.  
We close with two results related to this question. 

In the next two results, let $\mf g$ be a finite-dimensional simple Lie algebra with diagram automorphism $\sigma$, and let $L$ be the affine Kac-Moody algebra associated to $\mf g$ and $\sigma$.

\begin{proposition}\label{prop:ACC}
The algebras
$U(\mf g[t]^\sigma)$, $U(L)$, and  $U(L')$ satisfy the ascending chain condition (ACC) on completely prime ideals.
\end{proposition}
\begin{proof}
The proof of \cite[Proposition~6.4]{IS} works in our setting, appealing to Proposition~\ref{prop:localisedKM}, Theorem~\ref{thm:current}, and Theorem~\ref{ithm:U} as necessary.  We omit the details.
%
%
\begin{comment}
As noted in the proof of \cite[Proposition~6.4]{IS}, any ring with just-infinite growth satisfies ACC on completely prime ideals.   
Thus  the result holds for $U(\mf g[t]^\sigma)$ by Theorem~\ref{thm:current}.  

Let $U$ be either $U(L)$ or $U(L')$.
Consider an ascending chain $P_1 \subseteq P_2 \subseteq \cdots$ of completely prime ideals of $U$.
If $\bigcup P_n$ contains a nonzero element of $\kk[c]$, then as the $P_n$ are prime and $c$ is central, some $P_n$ contains an irreducible polynomial $c-\lambda \in \kk[c]$.
As $U/(c-\lambda)$ has just-infinite growth by Theorem~\ref{ithm:U},  the chain stabilizes by the first sentence of the proof.

So we may assume that each $P_n \cap \kk[c] = 0$.
As $P_n$ is prime, each $U/P_n$ is $\kk[c]$-torsionfree.
Thus if $P_n \neq P_{n+1}$, then $(P_{n+1}/P_n) \otimes_{\kk[c]}\kk(c) \neq 0$ and so
\[ P_n \otimes_{\kk[c]}\kk(c) \neq P_{n+1}  \otimes_{\kk[c]}\kk(c) .\]
Further, these ideals are completely prime as 
\[U \otimes_{\kk[c]}\kk(c)/ P_n  \otimes_{\kk[c]}\kk(c) \cong (U/P_n)  \otimes_{\kk[c]}\kk(c)\]
 is a domain.

Thus it suffices to show that $U \otimes_{\kk[c]}\kk(c) $ has ACC on completely prime ideals.
By Proposition~\ref{prop:localisedKM}, $U  \otimes_{\kk[c]}\kk(c) $ has just infinite growth.
Thus by the first part of the proof, $U  \otimes_{\kk[c]}\kk(c) $ satisfies the ACC on completely prime ideals.
\end{comment}
\end{proof}

\begin{remark}\label{rem:ACC}
Let $W_+ = t^2 \kk[t] \frac{d}{dt}$ be the {\em positive Witt algebra}.
It is shown in  \cite[Theorem~1.2]{IS} that $U(W_+)$ has just-infinite growth.
Further, by \cite[Theorem~1.5]{PS}, the symmetric algebra $S(W_+)$ satisfies the ascending chain condition on radical Poisson ideals.
These results are supporting evidence for the conjecture \cite[Conjecture~1.3]{PS} that the enveloping algebra $U(W_+)$ satisfies the ascending chain condition on two-sided ideals.

It is now known \cite{LSS} that symmetric algebras of (twisted) loop algebras satisfy the ascending chain condition on radical Poisson ideals, and this can easily be extended to symmetric algebras of affine Kac-Moody algebras.
Combining this result with Theorem~\ref{ithm:U} and Proposition~\ref{prop:ACC}, it is natural to ask if enveloping algebras of affine Kac-Moody algebras  satisfy the ascending chain condition on two-sided ideals.
(It is easy to see that these enveloping algebras are not left or right noetherian.)
This is the subject of ongoing research.  Note that this question is only really interesting for $U_0(L')$.
\end{remark}

 A ring $R$ is {\em Hopfian} if $R$ is not isomorphic to any proper quotient $R/J$ (equivalently, any epimorphism from $R \to R$ is an isomorphism).
 If $R$ satisfies the ascending chain condition on two-sided ideals, then $R$ must be Hopfian.  
 We do not know if enveloping algebras of affine Lie algebras satisfy this ACC, but it is a consequence of our growth results that they are Hopfian.
 Further, enveloping algebras of current algebras and central quotients of enveloping algebras of affine algebras satisfy the stronger Bassian property, where a ring $R$
 is {\em Bassian} if there is no injection of $R$ into any proper quotient $R/J$.

\begin{proposition}\label{prop:HB}
The algebras $U(\mf g[t]^\sigma)$, $U_0(L')$, and  $U_0(L)$ are Bassian and Hopfian. Further, $U(L)$ and $U(L')$ are Hopfian.
\end{proposition}

\begin{proof}
This proof is similar to the proof of \cite[Proposition~6.5]{IS}, but as it is fairly brief we give it here in full.

If $R$ has just infinite growth, then $\GK R/J < \GK R$ for any proper ideal $J$ of $R$, so $R$ cannot inject in $R/J$.  Thus the Bassian (and thus Hopfian) property for $U(\mf g[t]^\sigma)$ follows from Theorem~\ref{thm:current}, and for $U_0(L') $ and $U_0(L)$  it follows from Theorem~\ref{ithm:U}(0).

Let $U$ be either $U(L)$ or $U(L')$.
To show that $U$ is Hopfian,  let $f$ be a surjective endomorphism of $U$, with kernel $J$.   As $U/J \cong \operatorname{Im}(f) = U$ is torsionfree as a module over $\kk[c]$,
the complex
\[0 \to J \otimes_{\kk[c]} \kk(c)  \to U \otimes_{\kk[c]} \kk(c) \to U\otimes_{\kk[c]} \kk(c) \to 0\]
is exact.  Now by Proposition~\ref{prop:localisedKM}, we must have $J \otimes_{\kk[c]} \kk(c) = 0$, as otherwise a $\kk(c)$-algebra of finite GK-dimension would surject onto one of infinite GK-dimension.
As $U$ is $\kk[c]$-torsionfree (or by Theorem~\ref{ithm:U}(2)), $J = 0$.
\end{proof}

\bibliographystyle{amsalpha}
%\bibliography{biblio}

\newcommand{\etalchar}[1]{$^{#1}$}
\providecommand{\bysame}{\leavevmode\hbox to3em{\hrulefill}\thinspace}
\providecommand{\MR}{\relax\ifhmode\unskip\space\fi MR }
% \MRhref is called by the amsart/book/proc definition of \MR.
\providecommand{\MRhref}[2]{%
  \href{http://www.ams.org/mathscinet-getitem?mr=#1}{#2}
}
\providecommand{\href}[2]{#2}

\end{document}